\documentclass[10pt]{article}

\usepackage{graphics}
\usepackage{graphicx}

\usepackage[a4paper, left=35mm,right=35mm,top=34mm,bottom=34mm]{geometry}
\usepackage[utf8]{inputenc}
\usepackage[T1]{fontenc}
\usepackage[english]{babel}

\usepackage{enumerate}
\usepackage{graphicx}
\usepackage{hyperref}
\hypersetup{
    colorlinks=true,
    linkcolor=blue,
    filecolor=magenta,      
    urlcolor=cyan,
}
\usepackage{lipsum}
\usepackage{amsfonts}
\usepackage{graphicx}
\usepackage{epstopdf}
\usepackage{algorithmic}
\def\del  {\partial}
\def\eps{\varepsilon}

\def\R{\mathbb{R}}

\def\N{\mathbb{N}}

\def\dt{{\rm d}t}
 \def\dx{{\rm d}x}
 \def\dy{{\rm d}y}

\def\ds{{\rm d}s}

\usepackage{amsmath}
\usepackage{cite}
\usepackage{hyperref}
\usepackage{psfrag}
\usepackage{amssymb}
\usepackage{amsfonts}
\usepackage{dsfont}
\usepackage{bbm}

 \usepackage{dsfont}
 \usepackage{psfrag}
 \usepackage{color}
 \usepackage{tikz}
 \usepackage{subfigure}
\usepackage{mathtools,amsthm,amssymb,amsfonts}
\usepackage{algorithm}
\usepackage{algorithmic}
\makeatother
\theoremstyle{plain}
\def\del  {\partial}
\def\eps{\varepsilon}

\def\R{\mathbb{R}}

\def\N{\mathbb{N}}

\def\dt{{\rm d}t}
 \def\dx{{\rm d}x}
 \def\dy{{\rm d}y}

\def\ds{{\rm d}s}
\newtheorem{assumption}{Assumption}
\newtheorem{proposition}{\textbf{Proposition}}
\newtheorem{conjecture}{\textbf{Conjecture}}
\newtheorem{corollary}{\textbf{Corollary}}
\newtheorem{remark}{\textbf{Remark}}
\newtheorem{theorem}{\textbf{Theorem}}
\newtheorem{lemma}{\textbf{Lemma}}
\newtheorem{definition}{\textbf{Definition}}

\newtheorem{example}[theorem]{Example}

\usepackage{caption} 
\usepackage{fancyhdr}

\author{
  {\normalsize Adrien Dekkers}\thanks{CentraleSup\'elec, Universit\'e Paris-Saclay, France.}
  \and
  {\normalsize Anna Rozanova-Pierrat}\thanks{CentraleSup\'elec, Universit\'e Paris-Saclay, France
    (correspondence, anna.rozanova-pierrat@centralesupelec.fr).}
	\and
	{\normalsize Alexander Teplyaev}\thanks{University of Connecticut, USA
(correspondence, teplyaev@uconn.edu).}
  		}
\title{Mixed boundary valued problem for linear and nonlinear wave equations in domains with fractal boundaries}
\date{}

\begin{document}
\maketitle
\thispagestyle{fancy}

\begin{abstract}
\noindent The weak well-posedness, with the mixed
boundary conditions, of the strongly damped linear wave equation and of the non
linear Westervelt equation is proved in a large natural class of Sobolev admissible non-smooth domains. In the framework of uniform domains in $\R^2$ or $\R^3$ we also validate the approximation of the solution of the Westervelt
equation on a fractal domain by the solutions on the prefractals using the Mosco convergence of the corresponding variational forms.  
\end{abstract}


\section{Introduction}
We study the weak well-posedness of wave equations, such as the strongly damped wave equation and the nonlinear Westervelt equation, in the largest possible class of bounded domains with the mixed boundary conditions. This class of domains contains the irregular case of boundaries, including fractals.
 
 Actually the Westervelt equation is a well-known model~\cite{Westervelt} of non-linear acoustics,
 \begin{equation}\label{EqWestIn}
 \partial^2_t \phi-c^2\Delta \phi-\nu \Delta\partial_t \phi=\alpha \del_t \phi \partial^2_t \phi+f
 \end{equation}
 describing the nonlinear propagation of acoustical waves in the thermo-viscous media~\cite{Roz1,DEKKERS-2020-3,DEKKERS-2020-1}, for instance of the ultrasounds. Here, $c>0$ is the sound speed in the unperturbed homogeneous medium, $\nu$ is the viscosity of the medium (a strictly positive constant) and $\alpha>0$ is a nonlinearity constant. 
 The model comes from the compressible Navier-Stokes system by small perturbations of a constant medium state in the assumptions that the motion is potential and the viscosity properties of the medium are small (of the same order as the perturbations of the density, the pressure, and velocity). Typically, the unknown function $\phi$ in~\eqref{EqWestIn} is the velocity potential: $\mathbf{v}=-\nabla \phi$, but in this article, we consider its time derivative, allowing us to rewrite the nonlinear term in a more convenient for the mathematical analysis form:
 \begin{equation}\label{EqWest}
 \partial^2_t u-c^2\Delta u-\nu \Delta\partial_t u=\alpha u \partial^2_t u+\alpha (\partial_t u)^2+f.
 \end{equation}
 So, this time the velocity is related to solutions of~\eqref{EqWest} by the formula $\del_t \mathbf{v}=-\nabla u$.
 
 Irregular shapes, different to $C^2$-boundaries, are common to various geometry observed in nature~\cite{MANDELBROT-1983} and one of their typical models are fractal boundaries. Von Koch-like structures appear in nature as in the famous example~\cite{coast} of the coast of Britain.
 There are many other appearances of fractal domains in mathematics and 
physics, including 
the following papers most relevant to our work: 
\cite{SapovalGobron,BARDOS-2016,PhysRevLett.83.726,PhysRevLett.108.240602,HLTV,LNRG,FLV95,vdB,GS-C,Capitanelli,Capitanelli2,HinzMeinert}. For instance, it is known that the irregularity aspects are characteristical for cancer tumors, providing in addition an important vascularisation around. It makes possible to consider them as objects with boundaries of a higher ``thickness'', hence, mathematically, with a higher boundary dimension. Different ultrasound medical therapies and imaging thus could be areas of applications of the analytical studies of models of the nonlinear acoustics in bounded domains with irregular and possibly fractal boundaries. Fractals model naturally objects with a multi-scaling structure iterated up to infinity. This makes the fractal boundaries the most efficient in the heat exchanges~\cite{DE_GENNES-1982,ROZANOVA-PIERRAT-2012,BARDOS-2016}, optimal in the wave absorbtion~\cite{HDR,HINZ-2021-1} and the most stable structures under loads~\cite{HINZ-2021}. 
 Recently, it was shown that Lipschitz boundaries are not able to fulfill the minimum of the acoustical energy in the framework of a boundaries absorption problems~\cite{MAGOULES-2021} and that this minimum exists for more irregular, possibly, fractal shapes~\cite{HINZ-2021}.
 
 
%
The existence and regularity of the solutions of the Westervelt equation and their linear parts on regular domains ($i.e.$ of the wave and the strongly damped wave equations), typically with a $C^2$ boundary, are well known.
In addition, the solutions become more regular up to the boundary if the initial data 
are more regular. 
We can cite Evans~\cite{EVANS-2010} for the linear wave equation and Refs.~\cite{Kalt3,Kalt2,Kalt1,Kaltwer,Meyer} and the references therein for the strongly damped wave equation and the Westervelt equation with the Dirichlet boundary conditions. This approach, to use the regularity of the boundary, is not helpful even for the Lipschitz case. 
To be able to solve mixed boundary valued problems of partial differential equations (here the strongly damped wave equation and the non-linear Westervelt equation) in domains with non smooth or fractal boundaries it is important to describe a functional framework in which it is possible to consider the weak-well posedness of elliptic equations, in particular of the simplest one, the Poisson equation:
\begin{equation}\label{PoissonDir1}
\left\lbrace
\begin{array}{l}
-\Delta u=f \hbox{ in }\Omega,\\
u=0\hbox{ on }\Gamma_{D,\Omega},\\
\frac{\partial u}{\partial n}=0 \hbox{ on }\Gamma_{N,\Omega},\\
\frac{\partial u}{\partial n}+ a u=0 \hbox{ on }\Gamma_{R,\Omega},
\end{array}
\right.
\end{equation}
with $\partial\Omega=\Gamma_{D,\Omega}\cup \Gamma_{N,\Omega} \cup \Gamma_{R,\Omega}$.
 The results of Jones~\cite{JONES-1981} on $d-$sets and domains admitting $W^{k,p} $ extensions allow saying that, in dimension $2$, $(\epsilon,\delta)$-domains are the most general domains on which we can define traces and extensions of the Sobolev spaces and then solve the Poisson problem. However, it is not the case in $\R^3$ and higher dimensions. For this reason, thanks to optimal Sobolev extension results in $\R^n$ for $p>1$ found by~Haj\l{}as, Koskela and Tuominen~\cite{HAJLASZ-2008}, Arfi and Rozanova-Pierrat introduced in Ref.~\cite{ARFI-2017} a new type of domains with a possibly non-smooth boundary described by a $d$-set preserving Markov's local inequality called the admissible domains. The idea is to work in the class of domains, optimal by the Sobolev extension, and for which it is possible to define a surjective and continuous trace operator on the boundary, especially from $W^{1,2}(\Omega)$. 
 
 As in~\cite{HINZ-2021}, we use this concept (see also~\cite{ROZANOVA-PIERRAT-2020} for more detailed discussion) for boundaries described by the support of a finite upper regular Borel measure. 
 It allows us to consider not only $d$-set boundaries as in Ref.~\cite{ARFI-2017} but also boundaries consisting of different dimensional parts and which do not have a fixed dimension~\cite{JONSSON-2009,BIEGERT-2009}. As in review~\cite{ROZANOVA-PIERRAT-2020}, we call this class of domains the Sobolev admissible domains (see Section~\ref{secfirstresult}) and work on them to study the well-posedness of the Westervelt problem. 

 The most common examples of Sobolev admissible domains are domains with regular or Lipschitz boundaries, with a $d$-set boundaries such as von Koch fractals or with a ``mixed'' boundary presented, for instance, by a three-dimensional cylindrical domain constructed on a base of a two-dimensional domain with a $d$-set boundary~\cite{LANCIA-2010,ARXIV-CREO-2018}. 
 For instance, it could also be a uniform or $(\epsilon,\delta)$-domain with a boundary, which could be described by the support of a finite upper regular Borel measure.

Another important question is whether the solutions of the Poisson problem~(\ref{PoissonDir1}) belong to $C(\Omega)\cap L^{\infty}(\Omega) $ with an estimate of the form:
$$\Vert u\Vert_{L^{\infty}(\Omega)}\leq C \Vert f\Vert_{L^p(\Omega)}.$$
We generalize~\cite{Daners} and show this result for $(\varepsilon,\delta)$-domains and the Sobolev admissible domains.
This estimate is a key point to show that the solutions of our wave-type models are in $C(\Omega)\cap L^{\infty}(\Omega)$ but also to treat the nonlinear term in the Westervelt equation.
 We make attention to the fact that even for a Lipschitz boundary if the domain is not convex, the weak solution of the Poisson equation never belongs to $H^2(\Omega)$, but only to $H^1(\Omega)$, which restricts a lot the study of the Westervelt equation. The famous regularity result of Nyström~\cite{Nystrom,Nystromthese} illustrates the situation. Even for a positive source $f\in\mathcal{C}^\infty_0(\Omega)$ ($0$ means the compact support) the week solution of the Poisson homogeneous Dirichlet problem on the von Koch snowflake domain does not belong to $H^2(\Omega)$~\cite{Nystrom,Nystromthese}, but only to $H^2_{loc}(\Omega)$. This phenomenon is also called the interior regularity~\cite{EVANS-2010}. Therefore, we do not have the global $H^2(\Omega)$ regularity in the general framework of Sobolev admissible domains. To handle this difficulty, we start to study the linear part of the Westervelt equation, $i.e.$, the strongly damped wave equation.
Thus Section~\ref{secwpdampwavmix} is dedicated to the strongly damped wave equation and its weak well-posedness for mixed boundary conditions
using the Galerkin method~\cite{EVANS-2010}. 
A key point is the Poincar\'e inequality which we update for our case in Subsection~\ref{ApPoinc}. 
To obtain more regular solutions, we work in a subspace of $H^1(\Omega)$ defined by the domain of the Laplacian in the sense of $L^2$ or $L^p$. In particular, it means that in the absence of the global $H^2(\Omega)$-regularity~\cite{Nystrom,Nystromthese} of a weak solution $u\in H^1(\Omega)$, it is possible to ensure that $\Delta u\in L^2$ or $L^p$. This additional information is crucial to be able to treat in Section~\ref{secwpWesdirmix} the weak well-posedness of the Westervelt equation with mixed boundary conditions on three or two-dimensional Sobolev admissible domains. The control of the nonlinearity of a quadratic type does not allow to consider dimensions with $n\ge 4$. The proof method consists of applying an abstract theorem of Sukhinin~\cite{Sukhinin} as soon as possible to define an isomorphism between the space of the source term and the space of weak solutions of the linear problem. See also Ref.~\cite{Perso} for a similar application of Ref.~\cite{Sukhinin} in the framework of the strong well-posedness of the Cauchy problem for the Kuznetsov equation. A similar technique was used in~\cite{ARXIV-DEKKERS-2020} for the Dirichlet homogeneous and nonhomogenous boundary problems for the Westervelt equation on arbitrary and admissible domains, respectively. %

In Section~\ref{secconvprefrfr} we consider the question of the approximation of the weak solution of the Westervelt equation on a domain $\Omega$ with a fractal boundary by a sequence of weak solutions on the domains $\Omega_m$ with polyhedral boundaries converging to the fractal boundary in the limit. This time we work in the class of $(\eps,\infty)$ or uniform domains in $\mathbb{R}^n$.

We start in sub-Section~\ref{subsecOSC} by defining Assumptions \ref{a-face} and \ref{a-uniuni} on $\Omega$ 
and $\Omega_m$ so that they are all $(\eps,\infty)$-domains with a fixed $\eps$ independent on $m$. 
This property to be $(\eps,\infty)$ domain with the same $\eps$ is crucial to have the extension 
operators from $\Omega_m$ to $\R^n$ with norms independent on $m$ (see Subsection~\ref{subsectrace} 
and also~\cite[Thm~3.4]{Capitanelli2}). This uniform on $m$ boundness is important to be able to pass to the limit for $m\to +\infty$ in the Mosco convergence of the functionals corresponding to the weak formulations of the Westervelt mixed problem (see Subsection~\ref{subsecMosco}). %
In this way, 
Assumption~\ref{a-uniuni} 
ensures that for a fixed self-similar boundary of a domain in $\R^n$ the existence of a polyhedral boundary sequence of domains with the same $\eps$ as $\Omega$ itself. This generalizes the known two-dimensional approximation results for von Koch mixtures (for the definition, see Appendix~\ref{subsecexKoch}) of Refs.~\cite{Capitanelli,Capitanelli2}.
Thus, we introduce the trace and extension properties for the fixed $\Omega$ and $(\Omega_m)_{m\in \N^*}$ defined in Subsection~\ref{subsecOSC}.
In Subsection~\ref{subsecMosco} we establish the Mosco convergence of the functionals defined by the variational formulation for the Westervelt equation. In the presence of the nonlinear terms, the Mosco convergence result holds only in $\R^2$ or $\R^3$. Nevertheless, the Mosco convergence of the linear part holds in $\R^n$ for all $n\ge 2$. Finally, we finish by proving that the weak solutions $u_m$ on the prefractal approximate domains $\Omega_m$ converge weakly to the weak solution $u$ on the fractal domain (see Theorem~\ref{thmconv}), a method often uses in the case of shape optimization~\cite{MAGOULES-2021}. We notice that since our proof does not require any monotone assumption on $\Omega_m$, our approximation result works in particular for the so-called Minkowski fractal domain~\cite{SapovalGobron,PhysRevLett.83.726,PhysRevLett.108.240602},
and their $3$-dimensional analog.

To summarize, the rest of the paper is organized as follows. 
Section~\ref{secfirstresult} introduces the general functional framework of Sobolev admissible domains on which we update the Poincar\'e inequality (see sub-Section~\ref{ApPoinc}). In Section~\ref{SecPoisson}, noticing the well-posedness of the Poisson mixed problem and the properties of its spectral problem on the Sobolev admissible domains, we introduce the domain of the Laplacian in the sense of $L^2$ and of $L^p$ and generalize Daners' estimate for the Sobolev admissible domains (the proof is given for the completeness in Appendix~\ref{AppDaners}). In Section~\ref{secwpdampwavmix} we consider the weak well-posedness firstly of the mixed initial-boundary value problem for the strongly damped linear wave equation (sub-Section~\ref{subsecweakwpdampwavemix}) and then of the Westervelt equation (sub-Section~\ref{secwpWesdirmix}) both in the $L^2$ and $L^p$ frameworks on the Sobolev admissible domains. In Section~\ref{secconvprefrfr} we consider the approximation of the fractal problem for the Westervelt equation by prefractal problems with Lipschitz boundaries. In sub-Section~\ref{subsecOSC} we define the conditions on $\Omega$ and $\Omega_m$ in $\R^n$ such that they are all $(\eps,\infty)$-domains with a fixed $\eps$ independently on $m$. In sub-Section~\ref{subsectrace} we give the main trace and extension theorems with the uniform on $m$ estimates allowing to pass to the limit. In sub-Section~\ref{subsecMosco} we give the Mosco convergence result (Theorem~\ref{Mconv}) and the weak convergence of the prefractal weak solutions of the Westervelt equation to the fractal one for domains in $\R^2$ or $\R^3$ (Theorem~\ref{thmconv}). The example of a fractal boundary given by Koch mixtures is detailed in Appendix~\ref{subsecexKoch}.

 \section{Functional analysis framework and notations}\label{secfirstresult}
\subsection{Sobolev extension domains}\label{SubsFA1}
Let us start to define the Sobolev extension domains:
\begin{definition}[$W^{k,p}$-extension domains]\label{DefExtD}
 A domain $\Omega\subset \R^n$ is called a $W^{k,p}$-extension domain ($k\in \N^*$) if there exists a bounded linear extension operator $E: W^{k,p}(\Omega) \to W^{k,p}(\R^n)$. This means that for all $u\in W^{k,p}(\Omega)$ there exists a $v=Eu\in W^{k,p}(\R^n)$ with $v|_\Omega=u$ and it holds
 $$\|v\|_{W^{k,p}(\R^n)}\le C\|u\|_{W^{k,p}(\Omega)}$$
 with a constant $C>0$ independent of $u$ (depending only on $k$, $p$, $n$ and $\Omega$).
\end{definition}
It is known~\cite{JONES-1981} that the results of Calderon and Stein~\cite{CALDERON-1961,STEIN-1970} about Sobolev extension domains for domains with Lipschitz boundaries can be improved by the class of $(\eps,\delta)$-domains, or locally uniform domains, which in the bounded case are simply called uniform domains~\cite{HERRON-1991}.
\begin{definition}[$(\eps,\delta)$-domain~\cite{JONES-1981}]\label{DefEDD}
An open connected subset $\Omega$ of $\R^n$ is an $(\eps,\delta)$-domain, $\eps > 0$, $0 < \delta \leq \infty$, if whenever $(x, y) \in \Omega^2$ and $|x - y| < \delta$, there is a rectifiable arc $\gamma\subset \Omega$ with length $\ell(\gamma)$ joining $x$ to $y$ and satisfying
\begin{enumerate}
 \item[(i)] $\ell(\gamma)\le \frac{|x-y|}{\eps}$ and
 \item[(ii)] $d(z,\del \Omega)\ge \eps |x-z|\frac{|y-z|}{|x-y|}$ for $z\in \gamma$. 
\end{enumerate}
\end{definition}
The $(\eps,\delta)$-domains give the optimal class of Sobolev extension domains in $\R^2$ (see~\cite{JONES-1981} Theorem~3), but not in $\R^3$, where there exist Sobolev extension domains which are not $(\eps,\delta)$-domains. The extension constant in Definition~\ref{DefExtD} for an $(\eps,\infty)$-domain depends only on $k$, $p$, $n$ and $\eps$~\cite{JONES-1981}, and thus it is uniform for all $(\eps,\infty)$-domains with the same $\eps$.

Recently, the question about the optimal class of Sobolev extension domains in $\R^n$ was solved in terms of $n$-sets by~\cite{HAJLASZ-2008} for $W^{k,p}$-extension domains with $1<p<\infty$ and $k\in \N$ for domains in $\R^n$.
 For the completeness of the paper, we explain the notion of $d$-sets:
\begin{definition}[Ahlfors $d$-regular set or $d$-set~\cite{JONSSON-1984,JONSSON-1995,WALLIN-1991,TRIEBEL-1997}]\label{defdset}
Let $F$ be a Borel non-empty subset of $\R^n$. The set $F$ is is called a $d$-set ($0<d\le n$) if there exists a $d$-measure $\mu$ on $F$, $i.e.$ a positive Borel measure with support $F$ ($\operatorname{supp} \mu=F$) such that there exist constants 
$c_1$, $c_2>0$,
\begin{equation*}
 c_1r^d\le \mu(F\cap\overline{B_r(x)})\le c_2 r^d, \quad \hbox{ for } ~ \forall~x\in F,\; 0<r\le 1,
 \end{equation*}
where $B_r(x)\subset \R^n$ denotes the Euclidean ball centered at $x$ and of radius~$r$.
\end{definition}
 As~\cite[Prop.~1, p~30]{JONSSON-1984} all $d$-measures on a fixed $d$-set $F$ are equivalent, it is also possible to consider in Definition~\ref{defdset} the $d$-measure $\mu$ equal to the restriction of $d$-dimensional Hausdorff measure $m_d$ to $F$ ($m_d|_F$ is a $d$-measure on $F$ by~\cite[Thrm~1]{JONSSON-1984}).
%
 This in particular implies that a $d$-set $F$ has Hausdorff dimension $d$ in the neighborhood of each point of $F$~\cite[p.33]{JONSSON-1984}.
Definition~\ref{defdset} includes the case $d=n$, $i.e.$ $n$-sets. In $\mathbb{R}^n$ Lipschitz domains and domains with more regular boundaries are $n-$sets and their boundaries are $(n-1)-$sets. Using~\cite{JONSSON-1984,WALLIN-1991}, the $(\varepsilon,\delta)$ domains in $\mathbb{R}^n$ are $n-$sets: 
 \begin{equation}\label{EqNset}
 \exists c_\delta>0\quad \forall x\in \overline{\Omega}, \; \forall r\in]0,\delta[\cap]0,1] \quad \lambda(B_r(x)\cap \Omega)\ge C\lambda(B_r(x))=c_\delta r^n,
 \end{equation}
 where $\lambda(A)$ denotes the $n$-dimensional Lebesgue measure of a set $A$. This property is also called the measure density condition~\cite{HAJLASZ-2008}. Let us notice that an $n$-set 
$\Omega$ cannot be ``thin'' close to its boundary $\del \Omega$.
 At the same time~\cite{WALLIN-1991}, if $\Omega$ is an $(\eps,\delta)$-domain and $\del \Omega$ is a $d$-set ($d<n$), then $\overline{\Omega}=\Omega\cup \del \Omega$ is an $n$-set. A typical example of a $d$-set boundary it is the self-similar fractals as the von Koch fractals. 

To describe the optimal class of Sobolev extension domains in $\R^n$ for $1<p <\infty$, we cite the following result~\cite{HAJLASZ-2008}:
\begin{theorem}[Sobolev extension~\cite{HAJLASZ-2008}]\label{ThSExHaj}
 For $1<p <\infty$, $k=1,2,...$ a domain $\Omega\subset \R^n$ is a $W^k_p$-extension domain if and only if $\Omega$ is an $n$-set and $W^{k,p}(\Omega)=C_p^k(\Omega)$ (in the sense of equivalent norms).
\end{theorem}
In Theorem~\ref{ThSExHaj} the spaces $C_p^k(\Omega)$, $1< p<+\infty$, $k=1,2,...$ are the spaces of fractional sharp maximal functions,
\begin{multline*}
 C_p^k(\Omega)=\{f\in L^p(\Omega)|\\
 f_{k,\Omega}^\sharp(x)=\sup_{r>0} r^{-k}\inf_{P\in \mathcal{P}^{k-1}}\frac{1}{\lambda(B_r(x))}\int_{B_r(x)\cap \Omega}|f-P|\dy\in L^p(\Omega)\}
 \end{multline*}
with the norm $\|f\|_{C_p^k(\Omega)}=\|f\|_{L^p(\Omega)}+\|f_{k,\Omega}^\sharp\|_{L^p(\Omega)}$ and with the notation $\mathcal{P}^{k-1}$ for the space of polynomials on $\mathbb{R}^n$ of degree less or equal $k-1$.

From~\cite{JONES-1981} and~\cite{HAJLASZ-2008} we directly have~\cite{ARFI-2017}
\begin{corollary}
 Let $\Omega$ be a bounded finitely connected domain in $\R^2$ and $1<p<\infty$, $k\in \N^*$. The domain $\Omega$ is a $2$-set with $W_p^k(\Omega)=C_p^k(\Omega)$ (with norms' equivalence) if and only if $\Omega$ is an $(\eps,\delta)$-domain and its boundary $\del \Omega$ consists of a finite number of points and quasi-circles.
\end{corollary}

\subsection{Trace operator}\label{SubsFA2}
Once we know the optimal class of the Sobolev extension domains, we need to define the trace operator for elements of $W^{k,p}(\Omega)$ on the boundaries of these domains~\cite{ARFI-2017,ROZANOVA-PIERRAT-2020,HINZ-2021,HINZ-2021-1}.
As in~\cite{HINZ-2021,HINZ-2021-1}, we consider the Sobolev extension domains with compact boundaries defined by the support of a positive Borel measure $\mu$ on $\R^n$ ($\del \Omega=\operatorname{supp} \mu$), which in addition is \emph{upper $d$-regular} for a fixed $d>0$~\cite{AH96,FALCONER-1985}, $d\in]n-2,n[$: 
there is a constant $c_d>0$ such that 
\begin{equation}\label{EqMuUP}
\mu(B_r(x))\leq c_d r^d,\quad x\in \del \Omega,\quad 0<r\leq 1.
\end{equation}
Condition~\eqref{EqMuUP} implies that the Hausdorff dimension of the boundary $\dim_H \del \Omega \geq d$. If $\mu$ satisfies both bounds of Definition~\ref{defdset}, then 
$d\le \dim_H \del \Omega \le d$ and hence $\dim_H\del \Omega=d$ (the boundary is the $d$-set). 

This kind of boundary measure is rather general and could be, in particular cases, the Jonsson measures~\cite{JONSSON-1994,ROZANOVA-PIERRAT-2020}, the $d$-measures or a union of different measures of these types. For this general measure $\mu$, supported on a closed subset $\del \Omega\subset \mathbb{R}^n$, we define the corresponding Lebesgue spaces $L^p(\del \Omega,\mu)$~\cite{JONSSON-1994}.

In this general case of $\del \Omega$, $i.e.$ without the classical assumption of the regularity of the boundary, we cannot ensure that $C(\overline{\Omega})$ is dense in $W^{1,2}(\Omega)$ (only $C^\infty(\Omega)$ is still dense). 
Thus, we need to define what is a trace of $u\in W^{1,2}(\Omega)$ on the boundary in the general case for $n\ge 2$.


%
\begin{definition}[Pointwise trace]\label{deftrace}
Let $\Omega$ be a $W^{1,2}$-extension domain and $u\in W^{1,2}(\Omega)$. The trace operator $\operatorname{Tr}: W^{1,2}(\Omega)\to L^2(\del \Omega,\mu)$ on $\del \Omega$ is defined $\mu$-everywhere by $\operatorname{Tr} u:=\widetilde{g}$, where 
\begin{equation}\label{E:pointwiseredef}
\widetilde{g}(x)=\lim_{r\to 0}\frac{1}{\lambda(B_r(x))}\int_{B_r(x)}g(y)dy,\quad x\in \del \Omega,
\end{equation}
is the pointwise redefinition of an extension $g\in W^{1,2}(\mathbb{R}^n)$ of $u$.
\end{definition}
As it is noticed in~\cite[Remark~2(i)]{HINZ-2021-1} and~\cite[Section~5.1]{HINZ-2021}, since the boundary measure $\mu$ satisfies \eqref{EqMuUP} with $d\in]n-2,n[$, the set of points of $\partial\Omega$ where this limit exists is of full $\mu_{\partial\Omega}$-measure, \cite[Section 7]{AH96}. The independence of the chosen extension is proved in \cite[Theorem 6.1]{BIEGERT-2009}, another proof is given in \cite[Theorem 1]{WALLIN-1991}. In 
\cite[Theorem 1]{WALLIN-1991} it is shown that \eqref{E:pointwiseredef} is equivalent to
\begin{equation}\label{EqDefTraceM}
 \hbox{Tr} u(x):=\lim_{r\rightarrow 0} \frac{1}{\lambda(\Omega\cap B_r(x))}\int_{\Omega\cap B_r(x)} u(y)\;d\lambda.
\end{equation}
We can also notice that formally~\eqref{EqDefTraceM} is well-defined for $u\in L^1_{loc}(\Omega)$~\cite{WALLIN-1991}.

We give a trace result~\cite[Section~3]{HINZ-2021-1}, that for elements of $W^{1,2}(\Omega)$, or equivalently, of $H^1(\Omega)$, this limit exists $\mu$-a.e.. It follows, as for \cite[Theorem 5.1]{HINZ-2021}, from~\cite[Corollaries 7.3 and 7.4]{BIEGERT-2009} and the finiteness of the measure on $\partial\Omega$. The result uses \cite[Theorems 7.2.2 and 7.3.2]{AH96}.

\begin{theorem}\textbf{(\cite[Theorem~1]{HINZ-2021-1},\cite[Theorem 5.1]{HINZ-2021})}\label{ThGBesov}
Let $\Omega\subset \mathbb{R}^n$ be a $W^{1,2}$-extension domain. Suppose that $\mu$ is a Borel measure with $\operatorname{supp}\mu=\partial\Omega$ compact in $\mathbb{R}^n$ and such that \eqref{EqMuUP} holds with some $d\in ]n-2,n[$. 
 \begin{enumerate}
 \item[(i)]
 There are a compact linear operator $\operatorname{Tr}:W^{1,2}(\Omega) \to L^2(\partial\Omega,\mu)$ and a constant $c_{\mathrm{Tr}}>0$, depending only on $n$, $\Omega$, $d$ and $c_d$, such that 
 \[\left\|\operatorname{Tr} f\right\|_{L^2(\partial\Omega,\mu)}\leq c_{\operatorname{Tr}}\left\|f\right\|_{W^{1,2}(\Omega)},\quad f\in W^{1,2}(\Omega).\]
 If $\Omega$ is $(\eps,\infty)$-domain, then the constant $c_{\mathrm{Tr}}>0$ depends only on $n$, $\eps$, $d$ and $c_d$.
 Endowed with the norm 
 \[\left\|\varphi\right\|_{\operatorname{Tr}(W^{1,2}(\Omega))}:=\inf\{ \left\|g\right\|_{W^{1,2}(\Omega)}|\ \varphi=\mathrm{Tr}\:g\}\]
 the image $\operatorname{Tr}(W^{1,2}(\Omega))$ becomes a Hilbert space. 
 The embedding $$\operatorname{Tr}(W^{1,2}(\Omega))\subset L^2(\partial\Omega,\mu)$$ is compact.
 \item[(ii)] There is a linear operator $H_{\partial\Omega}:\operatorname{Tr}(W^{1,2}(\Omega)) \to W^{1,2}(\Omega)$ of norm one such that $\operatorname{Tr}(H_{\partial\Omega}\varphi)=\varphi$ for all $\varphi\in \operatorname{Tr}(W^{1,2}(\Omega))$.
 \item[(iii)] If $\Omega$ is bounded, then the norm $\|u\|_{W^{1,2}(\Omega)}$ on $W^{1,2}(\Omega)$ is equivalent to $$\|u\|_{\mathrm{Tr}}=\left( \int_\Omega |\nabla u|^2\dx +\int_{\del \Omega} |\mathrm{Tr}u|^2 d\mu \right)^\frac{1}{2} .$$ 
 \end{enumerate}
\end{theorem}
Actually, the linearity of $H_{\partial\Omega}$ follows from the linearity of $1$-harmonic extension operator on $L^2$-based spaces~\cite[Section~5.1]{HINZ-2021}. Since $\mu$ is Borel regular then $\operatorname{Tr}_{\del \Omega}(W^{1,2}(\Omega))$ is dense in $L^2(\del \Omega,\mu)$. For the proof of (iii) see for example~\cite[Proposition~3]{ARFI-2017} or Theorem 21A and Step 3 in its proof on pp.~247--248 of~\cite{ZEIDLER}.

In what following we denote the space of the image of the trace $\operatorname{Tr}_{\del \Omega}(W^{1,2}(\Omega))$ by $B(\del \Omega,\mu)$.
Knowing more information on the measure $\mu$, $i.e.$ on the regular properties of the boundary $\del \Omega$, it is possible to give the following caracterization of the space $B(\del \Omega,\mu)$:
\begin{enumerate}
 \item if $\del \Omega$ is a Lipschitz boundary, then $B(\del \Omega,\mu)=H^\frac{1}{2}(\del \Omega)$ and $\mu$ is $(n-1)$-dimensional Lebesgue measure on $\del \Omega$~\cite{LIONS-1972,MARSCHALL-1987};
 \item if $\mu$ is a $d$-dimensional measure with $n-2<d<n$, $i.e.$ $\del \Omega$ is a $d$-set, then $B(\del \Omega,\mu)$ is the Besov space $B^{2,2}_\beta(\del \Omega)$ with $\beta=1-\frac{n-d}{2}>0$~\cite{WALLIN-1991,JONSSON-1984};
 \item if $\mu$ in addition to~\eqref{EqMuUP} also satisfies for $n-2<d\le s<n$ for some constants $c_s>0$, $c_d>0$ the following conditions for all $x\in \del \Omega, \; r>0,\quad k\ge 1, \; 0<k r\le 1$
 \begin{align}
&\mu(B(x,kr))\le c_s k^s \mu(B(x,r)), \quad \mu(B(x,kr))\ge c_d k^d \mu(B(x,r)),\label{EqMuDs}
\end{align}
and for constants $c_1,\,c_2>0$ independent of $x$
\begin{align}
&c_1\leq \mu(B(x,1))\leq c_2,\quad x\in \del \Omega, \label{Eqnormalized}
\end{align}
$i.e.$ if $\del \Omega$ is more general as a fixed dimensional $d$-set, then $B(\del \Omega,\mu)$ is the Besov space $B_1^{2,2}(\del \Omega)$~\cite{JONSSON-1994}. This type of measures are also called Jonsson measures.
\end{enumerate}


In what follows we also use the generalization of Rellich-Kondrachov theorem on Sobolev extension domains~\cite{ARFI-2017}:
\begin{theorem}[Sobolev's embeddings]\label{thmsobolembadm}
Let $\Omega\subset\mathbb{R}^n$ be a bounded $W^{k,p}$-extension domain, $1<p<+\infty$, $k$, $l\in\mathbb{N}^*$. Then there hold the following compact embeddings
\begin{enumerate}
\item $W^{k+l,p}(\Omega)\subset\subset W^{l,p}(\Omega)$,
\item $W^{k,p}(\Omega)\subset \subset L^q(\Omega)$, 
\end{enumerate}
with $q\in [1,+\infty[$ if $kp=n$, $q\in [1,+\infty]$ if $kp>n$, and with $q\in \left[1,\frac{pn}{n-kp}\right[$ if $kp<n$.
Moreover if $kp<n$ we have the continuous embedding
$$W^{k,p}(\Omega) \subset L^{\frac{pn}{n-kp}}(\Omega).$$
\end{theorem}
\subsection{Sobolev admissible domains}
To simplify the notations, we use the notion of (Sobolev) admissible domains~\cite{ARFI-2017,
ROZANOVA-PIERRAT-2020,HINZ-2021}, allowing to ensure the continuity of the extension/trace operators from/to a domain and its boundary at the same time (see subsections~\ref{SubsFA1} and~\ref{SubsFA2}): 
\begin{definition}[Sobolev admissible domain]\label{defadmisdomain}
A domain $\Omega\subset \R^n$ is called a Sobolev admissible domain if it is a Sobolev extension domain, with a compact boundary $\del \Omega$ which is the support of a Borel measure $\mu$ satisfying \eqref{EqMuUP} for some $d$, $0\le n-2< d<n$.
 \end{definition}
By Theorem~\ref{ThGBesov} (see also~\cite[Theorem~5.1]{HINZ-2021}), the trace operator $\operatorname{Tr}:W^{1,2}(\Omega) \to L^2(\partial\Omega,\mu)$ is compact. By Sobolev extension domain properties (see for the geometrical caracteriszation Theorem~\ref{ThSExHaj}), for fixed $1<p<\infty$ and $k\in \N^*$ the extension operator $\mathrm{E}: W^{k,p}(\Omega)\to W^{k,p}(\R^n)$ with the right inverse $\mathrm{T}: W^{k,p}(\R^n)\to W^{k,p}(\Omega)$ are linear continuous (by~\cite{HAJLASZ-2008} the continuity constant depends only on $k$, $p$, $n$ and the constant $c_\delta$ from the definition of $n$-set~\eqref{EqNset}).

In what follows, we are interesting in the case $k=1$ and we write $H^1(\Omega)$ instead of $W^{1,2}(\Omega)$.
 \begin{example}
 An example of a Sobolev admissible domain could be a bounded domain of $\R^n$ with a boundary $\del \Omega$ equal to a finite disjoint union of parts $\Gamma_j$ which are $d_j$-sets respectively for $n-1\le d_j<n$ ($j=1,\ldots,m$). For instance, it is the case of a three-dimensional cylindrical domain constructed on a base of two-dimensional domain with a $d$-set boundary as considered for the Koch snowflake base in~\cite{LANCIA-2010,ARXIV-CREO-2018}. 
 \end{example}

Once $B(\del \Omega,\mu)$ is a Hilbert space, independently on the chosen boundary measure $\mu$ satisfying \eqref{EqMuUP}, it is possible to work with his topological dual space $B'(\del \Omega,\mu)$ and to understand the normal derivative on $\del \Omega$ as an element of $B'(\del \Omega,\mu)$ using the usual Green formula~\cite{LANCIA-2002,LANCIA-2003,ARXIV-CREO-2018,ROZANOVA-PIERRAT-2020,HINZ-2021}.

\begin{proposition}\label{PropGreen}\textbf{(Green formula)}
Let $\Omega$ be a Sobolev admissible domain in $\R^n$ ($n\ge 2$) with boundary $\del \Omega=\operatorname{supp} \mu$ satisfying \eqref{EqMuUP} with $n-2<d<n$. Then for all $u,\;v\in H^1(\Omega)$ with $\Delta u\in L^2(\Omega)$ it holds the Green formula
\begin{equation}\label{EqFracGreen}
 \langle \frac{\del u}{\del \nu}, 
 \mathrm{Tr}v\rangle _{(B'(\del \Omega,\mu), B(\del \Omega,\mu))}:=\int_\Omega v\Delta u\dx+\int_\Omega \nabla v \nabla u \dx.
 \end{equation}
 
 Equivalently, for any Sobolev admissible domain $\Omega$ the normal derivative of $u\in H^1(\Omega)$ with $\Delta u\in L^2(\Omega)$ 
 is defined by~Eq.~\eqref{EqFracGreen} as a linear and continuous functional on $B(\del \Omega,\mu)$.
\end{proposition}
The statement of Proposition~\ref{PropGreen} follows, from the surjective property of the continuous trace operator $\mathrm{Tr}_{\del \Omega}:H^1(\Omega)\to B(\del \Omega,\mu)$.

\subsection{Poincaré inequality}\label{ApPoinc}
As it is known that the boundary regularity does not important to have the Poincar\'e inequality in $W^{1,p}_0(\Omega)$ spaces, it also holds on bounded (at least in one direction, $i.e.$ for domains containning in a domain of the form $]a,b[\times \R^{n-1}$ for $a<b$) Sobolev admissible domains: %
\begin{theorem}[Poincar\'e inequality]\label{inegPoinc}
Let $\Omega\subset\mathbb{R}^n$ with $n\geq 2$ be a bounded (at least in one direction) 
 domain. For all $u\in W^{1,p}_0(\Omega)$ with $1\leq p<+\infty$, there exists $C>0$ depending only on $\Omega$, $p$ and $n$ such that
$$\Vert u\Vert_{L^p(\Omega)}\leq C \Vert \nabla u\Vert_{L^p(\Omega)} .$$
Therefore the semi-norm $\Vert .\Vert_{W^{1,p}_0(\Omega)}$, defined by $\Vert u\Vert_{W^{1,p}_0(\Omega)}:=\Vert \nabla u\Vert_{L^p(\Omega)}$, is a norm which is equivalent to $\Vert .\Vert_{W^{1,p}(\Omega)}$ on $W^{1,p}_0(\Omega)$.

Moreover, if $\Omega$ is a bounded Sobolev extension domain and $1<p<+\infty$, for all $u\in W^{1,p}(\Omega)$ there exists $C>0$ depending only on $\Omega$, $p$ and $n$ such that
$$\left\Vert u-\frac{1}{\lambda(\Omega)}\int_{\Omega} u\;d\lambda\right\Vert_{L^p(\Omega)}\leq C \Vert \nabla u\Vert_{L^p(\Omega)} .$$
\end{theorem}
\begin{proof}
The result for $u\in W^{1,p}_0(\Omega)$ comes from the boundness of $\Omega$. The result for $u\in W^{1,p}(\Omega)$ comes from the compactness of the embedding $W^{1,p}(\Omega)\subset\subset L^p(\Omega)$ from Theorem~\ref{thmsobolembadm} and following for instance the proof in Ref.~\cite{EVANS-2010} (see section 5.8.1 Theorem 1).
\end{proof}
Let us denote by $\mathcal{H}^{n-1}$ the $(n-1)$-dimensional Hausdorff measure.
We introduce the space $V_{\Gamma}(\Omega)$ for a domain $\Omega$ with a non trivial closed part of boundary $\Gamma\subset \del \Omega$ ($i.e.$ $\mathcal{H}^{n-1}(\Gamma)>0$):
\begin{equation}\label{eqVOmega}
 V_\Gamma(\Omega):=\lbrace u\in H^1(\Omega)\vert\; Tr_{\Gamma} u=0\rbrace.
\end{equation}

 Let us give two results on the Poincar\'e's inequality on $(\varepsilon,\delta)$-domain which we use in Section~\ref{secconvprefrfr}.
 
 \begin{figure}[!ht]
\begin{center}
\psfrag{aa}{$a)$}
 \psfrag{bb}{$b)$}
 \psfrag{G}{$\Gamma$}
 \psfrag{Ge}{$\Gamma^*$}
 \psfrag{O}{$\Omega$}
 \psfrag{Oe}{$\Omega^*$}
 \psfrag{V}{$V$}
\includegraphics[scale=0.35]{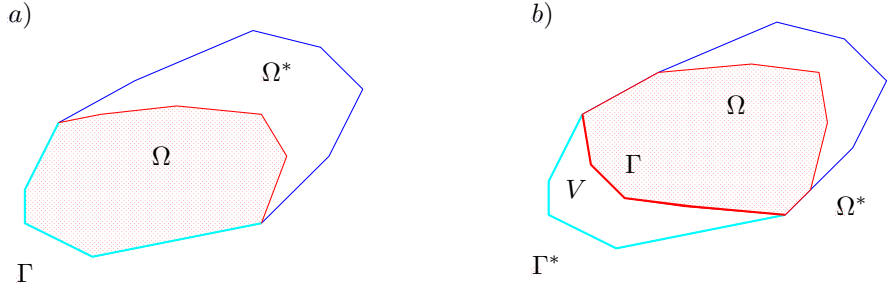} %
\vspace*{8pt}
\caption{\label{FigIlDom} Illustration for two possible cases treated in Theorem~\ref{thmPoinc2}. The case $a)$ corresponds to the case $\Gamma=\Gamma^*\subset\del \Omega \cap \partial\Omega^*$ and the case $b)$ to the case when $\Gamma$ and $\Gamma^*$ are the same the starting and the ending points. Each time $\Omega$ is the dots-filled area.}
\end{center}
\end{figure}

 \begin{theorem}\label{thmPoinc2}
 Let $\Omega\subset\Omega^*\subset\mathbb{R}^n$ be two bounded $(\varepsilon,\delta)$-domains (uniform domains) such that either 
 $$\Gamma=\Gamma^*\subset \partial\Omega\cap\partial\Omega^* \quad \hbox{with} \quad \mathcal{H}^{n-1}(\Gamma)>0,$$
 or $$ \Gamma \subset \partial\Omega,\quad \Gamma^*\subset\partial\Omega^*\quad \hbox{with} \quad \partial\Gamma=\partial \Gamma^*$$ 
 and $\Gamma\cup \Gamma^*$ defines the closed boundary of an open bounded set $V$ ($\partial V=\Gamma\cup \Gamma^*$) which $V \subset \Omega^*\setminus \Omega$ (see~Fig.~\ref{FigIlDom}).
 Then it holds the Poincar\'e inequality for all $u\in V_\Gamma(\Omega)$
 $$\Vert u\Vert_{L^2(\Omega)}\leq C \Vert \nabla u\Vert_{L^2(\Omega)}$$
 with $C>0$ depending only on $\varepsilon$, $\delta$ and the constant of the Poincar\'e inequality on $V_{\Gamma^*}(\Omega^*)$.
 \end{theorem}
 \begin{proof}
 By Theorem~\ref{inegPoinc} the Poincar\'e inequality holds on $V_\Gamma(\Omega)$ (see also \cite[Lemma~1 (ii)]{HINZ-2021-1}).
 Let us consider the following space
 $$u\in W(\Omega):=\lbrace u\in \mathcal{D}'(\Omega)\vert \;\Vert\nabla u\Vert_{L^2(\Omega)}<+\infty\rbrace.$$
 If $u\in V_\Gamma(\Omega)$ then obviously $u\in W(\Omega)$. In addition, according to Ref.~\cite{JONES-1981} (see also Refs.~\cite{AHMNT,Rogers}) $W(\Omega)$ admits a linear continuous extension to $W(\mathbb{R}^n)$, denoted by $\Lambda$, whose norm only depends of $\varepsilon$, $\delta$ and of $n$: 
 $$\Vert \nabla \Lambda u\Vert_{L^2(\mathbb{R}^n)}\leq C(\varepsilon,\delta,n) \Vert\nabla u\Vert_{L^2(\Omega)}$$
 and, as a consequence,
$$ \Vert \nabla \Lambda u\vert_{\Omega^*}\Vert_{L^2(\Omega^*)}\leq C(\varepsilon,\delta,n) \Vert\nabla u\Vert_{L^2(\Omega)}.$$
Let us start to consider the first case corresponding to the point a) on Figure~\ref{FigIlDom}.
By the definition of the extension $\Lambda u=u$ on $\Omega$, and by the analogous argument as in~\cite[Appendix~B]{HINZ-2021-1}, we have
$$Tr_{\Gamma} \Lambda u\vert_{\Omega^*}=Tr_{\Gamma} u=0.$$
Thus we can consider $\Lambda u\vert_{\Omega^*}\in V_\Gamma(\Omega^*)$ and by the Poincar\'e inequality on $V_\Gamma(\Omega^*)$ we have
$$\Vert \Lambda u\vert_{\Omega^*}\Vert_{L^2(\Omega^*)}\leq C(\Omega^*) \Vert \nabla \Lambda u\vert_{\Omega^*}\Vert_{L^2(\Omega^*)}.$$
To conclude we just notice that, as $\Omega$ and $\Omega^*$ are bounded, it holds
$$\Vert u\Vert_{L^2(\Omega)}\leq \Vert \Lambda u\vert_{\Omega^*}\Vert_{L^2(\Omega^*)}.$$
In the case b) of Figure~\ref{FigIlDom} we define $v$ on $\Omega^*$ by
$$v\vert_{\Omega^*\setminus \overline{V}}=\Lambda u\vert_{\Omega^*\setminus \overline{V}}\quad \hbox{and} \quad v\vert_{V\cup\Gamma}=0.$$ 
Then $v\in V_{\Gamma^*}(\Omega^*)$ and 
$$ \Vert \nabla v \Vert_{L^2(\Omega^*)}\leq C(\varepsilon,\delta,n) \Vert\nabla u\Vert_{L^2(\Omega)}.$$
By the Poincar\'e inequality on $V_{\Gamma^*}(\Omega^*)$ we have
$$\Vert v\Vert_{L^2(\Omega^*)}\leq C(\Omega^*) \Vert \nabla v \Vert_{L^2(\Omega^*)}$$
and as
$$\Vert u\Vert_{L^2(\Omega)}\leq \Vert v\Vert_{L^2(\Omega^*)},$$
 this finishes the proof.
 \end{proof}
 \section{Remarks on the Poisson equation with the mixed boundary conditions}\label{SecPoisson} 
 We start now to apply the introduced framework of Sobolev admissible domains for the mixed boundary valued problem for the Poisson equation. These preliminary discussion is crucial for the properties of the waves problems constructed on it.
 
 Let $\Omega$ be a bounded Sobolev admissible domain in $\mathbb{R}^n$, $n\ge 2$ with a boundary $\del \Omega=\operatorname{supp} \mu$. 
In all the sequel of this article, we suppose that its boundary $\partial\Omega=\Gamma_D\cup\Gamma_N\cup\Gamma_R$ is a disjoint union of three types of boundaries (corresponding to the Dirichlet, the Neumann, and the Robin boundary conditions respectively), each a Borel set and, at least $\Gamma_D$ and $\Gamma_R$, of positive measure $\mu$.
 We denote by $V(\Omega)$ the Hilbert subspace of $H^1(\Omega)$ (in Subsection~\ref{ApPoinc} it corresponds to $V_{\Gamma_D}(\Omega)$, but here we simplify the notation) 
 \begin{equation}\label{EqV}
 V(\Omega)=\{u\in H^1(\Omega)|\; \operatorname{Tr}u|_{\Gamma_D}=0\}
 \end{equation}
endowed with the following norm
\begin{equation}\label{normeqVOmega}
\Vert u\Vert_{V(\Omega)}^2=\int_{\Omega}|\nabla u|^2\;dx+a\int_{\Gamma_R}|Tr_{\partial\Omega}u|^2 {\rm d} \mu,
\end{equation}
associated to the inner product
$$(u,v)_{V(\Omega)}=\int_{\Omega}\nabla u\; \nabla v\;dx+a\int_{\Gamma_R}Tr_{\partial\Omega}u\;Tr_{\partial\Omega}v {\rm d} \mu.$$
 Thanks to Theorem~\ref{ThGBesov} 
 the norm $\Vert.\Vert_{V(\Omega)}$ is equivalent to the usual norm $\Vert .\Vert_{H^1(\Omega)}$ on $V(\Omega)$. 

 On $\Omega$ we consider the mixed boundary problem for the Poisson equation~\eqref{PoissonDir1} with a fixed $a>0$
 in the following weak sense:
 \begin{equation}\label{EqVFPoisson}
 \forall v\in V(\Omega) \quad (u,v)_{V(\Omega)}=(f,v)_{L^2(\Omega)}.
 \end{equation}
 Then (see for more details~\cite{ARFI-2017}) for all $f\in L^2(\Omega)$ and $a>0$ the Poisson problem~(\ref{PoissonDir1})
 has a unique weak solution $u\in V(\Omega)$. Furthermore, the mapping $ f\mapsto u$ is a compact linear operator from $L^2(\Omega)$ to $V(\Omega)$ with the estimate
 $$\Vert u\Vert_{V(\Omega)}\leq C(\Omega)\Vert f\Vert_{L^2(\Omega)} .$$

%

%

 Let us consider the corresponding spectral problem. We say that $\lambda\in \mathbb{C}$ is an eigenvalue of the Poisson problem~(\ref{PoissonDir1}) associated to the eigenfunction $u\in V(\Omega)$ with $\|u\|_{L^2(\Omega)}=1$, which is a weak solution of the following variational formulation
\begin{equation}\label{EqVFPoissonVP}
 \forall v\in V(\Omega)\quad \int_{\Omega}\nabla u\nabla v\dx+a \int_{\Gamma_R}Tr_{\partial\Omega}u\;Tr_{\partial\Omega}v {\rm d} \mu=\int_{\Omega} \lambda u v\dx. 
\end{equation}
 Thanks to the compactness by Theorem~\ref{ThGBesov} of the trace $\operatorname{Tr}: V(\Omega)\to L^2(\del \Omega)$ and of the inclusion $V(\Omega)\to L^2(\Omega)$ and by the assumption that $a>0$ is real, 
 we have the usual properties of the spectral problem associated with~(\ref{PoissonDir1}).
\begin{theorem}[Spectral Poisson mixed problem]\label{thmeigenfuncLaprob}
Let $\Omega$ be a bounded Sobolev admissible domain in $\mathbb{R}^n$ ($n\geq2$) and $a>0$. Then the operator $-\Delta$ associated with the spectral problem~\eqref{EqVFPoissonVP} is self-adjoint positive operator on the Hilbert space $V(\Omega)$ with a countable
%
number of real, strictly positive, eigenvalues of finite multiplicity, which can be ordered in a sequence
$$0<\lambda_1\leq\lambda_2\leq\lambda_3\leq\cdots \lambda_k\leq \ldots, \quad \lambda_k\rightarrow +\infty\hbox{ when }k\rightarrow +\infty,$$
 and the corresponding eigenfunctions $(w_k)_{k\in \N^*}\subset V(\Omega)$ form a basis of $V(\Omega)$ and an orthonormal basis of $L^2(\Omega)$.
\end{theorem}
\begin{remark}
For the problems with mixed-type boundary conditions involving the Robin type part of the boundary, it is crucial to work in the class of (bounded) Sobolev admissible domains. Indeed, the Sobolev extension property ensures the compactness of the embedding of $V(\Omega)$ into $L^2(\Omega)$, and the compactness of the trace operator $\operatorname{Tr}: V(\Omega) \to L^2(\Gamma_R,\mu)$ follows from the boundary properties, $i.e.$ from the properties of the Borel measure $\mu$. However, in the case $\del \Omega=\Gamma_D$ with the homogeneous Dirichlet boundary condition, as usual, it is possible to consider arbitrary bounded domains. 
\end{remark}
As $\Omega$ is a bounded domain, we have $L^p(\Omega)\subset L^2(\Omega)$ if $p\geq 2$, and consequently it is also possible to take $f\in L^p(\Omega)$ and consider the weak solutions in $V(\Omega)\cap L^p(\Omega)$ in the sense of~(\ref{EqVFPoisson}). Let us also notice~\cite{Daners} that for $p\ge 2>\frac{2n}{n+1}$ the space $V(\Omega)\cap L^p(\Omega)$ is dense in $L^p(\Omega)$. Therefore, there is the following generalization of the domain of the Laplacian in the $L^p$ framework~\cite{Daners}: 
\begin{definition}[Laplacian domain in $L^p$~\cite{Daners}]\label{defdomLapl}
Let $\Omega$ be a bounded Sobolev admissible domain and $p\geq 2$. We define the Laplacian operator associated with the mixed boundary problem~\eqref{EqVFPoisson}
\begin{align*}
-\Delta:\mathcal{D}_p(-\Delta)\subset V(\Omega)&\rightarrow L^p(\Omega) \\
u&\mapsto -\Delta u
\end{align*}
with the dense domain 
$$\mathcal{D}_p(-\Delta)=\{u\in V(\Omega)\cap L^p(\Omega)|\; -\Delta u\in L^p(\Omega), \hbox{ i.e. } \exists f\in L^p(\Omega) \hbox{ such that it holds~(\ref{EqVFPoisson})}\}$$
and introduce the notation $\Vert u\Vert_{ \mathcal{D}_p(-\Delta)}:=\Vert \Delta u\Vert_{L^p(\Omega)}=\|f\|_{L^p(\Omega)}$ for $u\in \mathcal{D}_p(-\Delta)$.
\end{definition}
By~\cite{Daners} the operator $-\Delta$ on $\mathcal{D}_p(-\Delta)$ is the $L^p$-realisation of the Laplacian for mixed boundary conditions, which can also be viewed as the generator of the associated heat semigroup on $L^p(\Omega)$~\cite[Theorem~6.1]{Daners}. In particular~\cite[Corollary~5.5]{Daners}, since $\Omega$ is a bounded Sobolev admissible domain, the spectrum of the operator $-\Delta$ on $\mathcal{D}_p(-\Delta)$ does not depend on the choice of $p\ge 1$.
%
%

The $L^p$-framework for the Poisson problem~(\ref{PoissonDir1}) is in particular important for the study of the continuity of its solution~\cite{Daners}. We directly update the result from Ref.~\cite{Daners} for the bounded Sobolev admissible domains with an $(n-1)-$set boundary:
\begin{theorem}\label{thmDaners}
Let $p>n$, $a>0$ and $\Omega$ be a bounded Sobolev admissible domain in $\mathbb{R}^n$ ($n=2$ or $3$) with a closed $(n-1)-$set boundary $\partial\Omega$. Let $u\in \mathcal{D}_p(-\Delta)$ be the unique solution of the Poisson problem~(\ref{PoissonDir1}) for $f\in L^p(\Omega)$. Then 
$$\Vert u\Vert_{L^{\infty}(\Omega)}\leq C \max\left(1,\frac{1}{a}\right) \Vert f\Vert_{L^p(\Omega)}.$$
\end{theorem}
Moreover, for all bounded Sobolev admissible domains, we improve Theorem~\ref{thmDaners} using the following result:
\begin{proposition}\label{propapl6}
Let $\Omega$ be a bounded Sobolev admissible domain in $\mathbb{R}^n$ ($n=2$ or $3$), then for all $u\in V(\Omega)$ the following estimate holds
\begin{equation}\label{L6est}
\Vert u\Vert_{L^6(\Omega)}\leq C \Vert \nabla u\Vert_{L^2(\Omega)},
\end{equation}
where $C>0$ is a constant depending only on $\Omega$. In addition, if $\Omega$ is an $(\varepsilon,\delta)-$domain, then $C>0$ depends only on $\varepsilon$, $\delta$, $n$ and the constant in the Poincar\'e inequality on $V(\Omega)$.
\end{proposition}
\begin{proof}
If $\Omega$ is a Sobolev admissible domain, then by Theorem~\ref{thmsobolembadm}, as $n=2$ or $3$, we have by the Sobolev embedding for $u\in V(\Omega)$
$$\Vert u\Vert_{L^6(\Omega)}\leq C\Vert u\Vert_{H^1(\Omega)}$$
and by the Poincar\'e inequality on $V(\Omega)$
$$\Vert u\Vert_{H^1(\Omega)}\leq C(\Omega)\Vert \nabla u\Vert_{L^2(\Omega)}.$$

Now let us treat the case when $\Omega$ is a $(\varepsilon,\delta)-$domain.
According to Ref.~\cite{JONES-1981}, as $\Omega$ is an $(\varepsilon,\delta)$-domain, we have a continuous extension operator
$E_{\Omega}:H^1(\Omega)\rightarrow H^1(\mathbb{R}^n),$
whose norm depends only on $\varepsilon$, $\delta$ and on $n$.
As $n=2$ or $3$ we have the continuous Sobolev embedding
$H^1(\mathbb{R}^n)\subset L^6(\mathbb{R}^n)$,
whose norm only depends on $n$.
Considering the continuous restriction (of norm equal to $1$)
$L^6(\mathbb{R}^n)\subset L^6(\Omega),$
we finally have the estimate
$$\Vert u\Vert_{L^6(\Omega)}\leq C \Vert u\Vert_{H^1(\Omega)}$$
where $C>0$ depends only on $\varepsilon$, $\delta$ and on $n$.
But $u\in V(\Omega)$, so the application of the Poincar\'e inequality allows to conclude.
\end{proof}
Thus we prove the general case
\begin{theorem}\label{thmlinfestL2}
Let $\Omega$ be a bounded Sobolev admissible domain in $\mathbb{R}^n$ ($n=2$ or $3$), $a>0$, $f\in L^2(\Omega)$ and $u\in V(\Omega)$ be the weak solution of~(\ref{PoissonDir1}) in the sense of the variational formulation~\eqref{EqVFPoisson}. Then it holds the estimate
$$\Vert u\Vert_{L^{\infty}(\Omega)}\leq C \Vert f\Vert_{L^2(\Omega)},$$
where the constant $C>0$ depends only on $\Omega$. If in addition $\Omega$ is an $(\varepsilon,\delta)-$domain, then the constant $C$ depends only on $\varepsilon$, $\delta$, $n$ and on the constant from the Poincar\'e inequality on $V(\Omega)$, but not on $a$.
\end{theorem}
The proof is a simplified variant of the proof of Theorem~4.1 of Ref.~\cite{Daners}. It is given for the completeness of the article in Appendix~\ref{AppDaners}.

\section{Well posedness of the damped linear wave equation}\label{secwpdampwavmix}
\subsection{Well posedness and $L^2$ regularity}\label{subsecweakwpdampwavemix}
In this subsection we suppose that $\Omega$ is a bounded Sobolev admissible domain in $\mathbb{R}^n$ on which we consider the following linear strongly damped wave equation in the previous framework of mixed boundary conditions:
\begin{equation}\label{dampedwaveeqmix}
\left\lbrace
\begin{array}{l}
u_{tt}-c^2 \Delta u- \nu \Delta u_t=f\;\hbox{ on }\;]0,+\infty[\times\Omega,\\
u=0 \hbox{ on }\Gamma_D\times[0,+\infty[,\\
\frac{\partial}{\partial n}u=0\hbox{ on }\Gamma_N\times[0,+\infty[,\\
\frac{\partial}{\partial n}u+a u=0\hbox{ on }\Gamma_R\times[0,+\infty[,\\
u(0)=u_0,\;\;u_t(0)=u_1 \hbox{ in } \Omega.
\end{array}
\right.
\end{equation}
We are looking for weak solutions of system~(\ref{dampedwaveeqmix}) in the following sense:
\begin{definition}\label{defweakdampwav}
For $f\in L^2([0,+\infty[;L^2(\Omega))$, $u_0\in V(\Omega)$, and $u_1\in L^2(\Omega)$, where $V(\Omega)$ defined in~(\ref{eqVOmega}),
we say that a function $u\in L^2([0,+\infty[;V(\Omega))$ with $\partial_t u \in L^2([0,+\infty[;V(\Omega))$ and $\partial^2_t u\in L^2([0,+\infty[;H^{-1}(\Omega)$ is a weak solution of problem~(\ref{dampedwaveeqmix}) if 
for all $v\in L^2([0,+\infty[;V(\Omega))$
\begin{equation}\label{varformdampedwaveeqdirhom}
\int_0^{+\infty}\langle u_{tt},v\rangle_{(H^{-1}(\Omega),V(\Omega))} +c^2 (u,v)_{V(\Omega)}+\nu ( u_t, v)_{V(\Omega)} ds=\int_0^{+\infty} (f,v)_{L^2(\Omega)}ds,
\end{equation}
with $ u(0)=u_0$ and $u_t(0)=u_1.$
\end{definition}
To prove the existence and uniqueness of such a weak solution, we use the Galerkin method and follow~\cite[p.~379--387]{EVANS-2010} using the fact that the Poincar\'e inequality stays true on $V(\Omega)$. 
To perform the Galerkin method we select functions $w_k=w_k(x)$, $k\in \N^*$ as the normalized eigenfunctions of the operator $-\Delta$ on $\Omega$ with the mixed boundary conditions, defined in Theorem~\ref{thmeigenfuncLaprob}: 
\begin{equation*}\label{eigenfuncDelta}
-\Delta w_k=\lambda_k w_k \hbox{ in a weak sense as }\forall w\in V(\Omega) \quad ( w_k, w)_{V(\Omega)}=\lambda_k (w_k,w)_{L^2(\Omega)}
\end{equation*}
and define then for a fixed $m\in \N^*$ the finite approximation of $u$ by
\begin{equation}\label{Galerksol}
u_m(t):=\sum_{i=1}^m d^k_m(t) w_k,
\end{equation}
where the coefficients $d^k_m(t)\in H^2(]0,+\infty[)$, $t \geq 0,\;k=1,...,m$ satisfy
\begin{align}
d^{k}_m(0)&=(u_0,w_k)_{L^2(\Omega)}\in\mathbb{R}\; k=1,...,m,\label{Galerkinit1}\\
\partial_t d^{k}_m(0)&=(u_1,w_k)_{L^2(\Omega)}\in\mathbb{R}\;k=1,...,m\label{Galrkinit2}
\end{align}
 and $u_m$ for $t \geq 0,\;k=1,...,m$ solves
\begin{equation}\label{Galerkdampedeq}
(\partial^2_t u_m,w_k)_{L^2(\Omega)}+c^2 ( u_m, w_k)_{V(\Omega)}+\nu ( \partial_t u_m, w_k)_{V(\Omega)}=(f,w_k)_{L^2(\Omega)}.
\end{equation}
As the rest of the proof is standard and repeat a lot~\cite{EVANS-2010} it is omitted but can be found in~\cite{dekthese}. Let us focus now on the regularity of such a solution.
For the weak solution of the damped wave equation problem~(\ref{dampedwaveeqmix}), satisfying Definition~\ref{defweakdampwav}, we have the following regularity results (for the proof see~\cite{dekthese}):
\begin{theorem}\label{dampedwaveeqregint1}
Let $\Omega$ be a Sobolev admissible bounded domain in $\mathbb{R}^n$ ($n\geq2$). Then there exists 
the unique weak solution $u$ of the strongly damped wave equation problem~(\ref{dampedwaveeqmix}) in the sense of Definition~\ref{defweakdampwav}.
Moreover,
\\(i) in addition $u$ has the following regularity
$$u\in L^{\infty}([0,+\infty[;V(\Omega)),\;\;\partial_t u\in L^{\infty}([0,+\infty[;L^2(\Omega))$$
and satisfies the estimate 
\begin{eqnarray}
\operatorname{ess} \sup_{t \geq 0}(\Vert u(t)\Vert_{V(\Omega)} && +\Vert\partial_t u(t)\Vert_{L^2(\Omega)})+\int_0^{+\infty} \Vert \partial_t u(s)\Vert_{V(\Omega)}\;ds+ \Vert\partial^2_t u\Vert_{ L^2([0,+\infty[;H^{-1}(\Omega))}\nonumber\\
&&\leq C ( \Vert f\Vert_{L^2([0,+\infty[;L^2(\Omega))}+\Vert u_0\Vert_{H^1_0(\Omega)}+\Vert u_1\Vert_{L^2(\Omega)}).\nonumber
\end{eqnarray}
(ii) If the initial data are taken more regular
$$ u_0\in \mathcal{D}_2(-\Delta),\;\;u_1\in V(\Omega),$$ 
 where $\mathcal{D}_2(-\Delta)$ comes from Definition~\ref{defdomLapl} with $p=2$, then in addition to the previous point the weak solution satisfies 
\begin{eqnarray}
&&\partial_t u\in L^{\infty}([0,+\infty[;V(\Omega)), \quad \partial^2_t u\in L^{2}([0,+\infty[;L^2(\Omega)),\nonumber\\
&&\Delta u \in L^{\infty}([0,+\infty[;L^2(\Omega))\cap L^2( [0,+\infty[;L^2(\Omega)),\nonumber\\
&&\Delta \partial_t u \in L^2( [0,+\infty[;L^2(\Omega))\nonumber
\end{eqnarray}
with the estimates
\begin{eqnarray}
\operatorname{ess}\sup_{t\geq 0}(\Vert\Delta u(t)\Vert_{L^2(\Omega)}^2&&+ \Vert \partial_t u(t)\Vert_{V(\Omega)}^2 )+\int_0^{\infty}\Vert\Delta\partial_t u(s)\Vert_{L^2(\Omega)}^2 \;ds \nonumber\\
\leq && C ( \Vert f\Vert_{L^2([0,+\infty[;L^2(\Omega))}^2+\Vert\Delta u_0\Vert_{L^2(\Omega)}^2+\Vert u_1\Vert_{V(\Omega)}^2)\label{aprioriglobdampedwave}
\end{eqnarray}
and 
\begin{equation}\label{aprioriglobdampedwavebis}
\int_0^{+\infty} \Vert \Delta u(s)\Vert_{L^2(\Omega)}^2\;ds\leq C ( \Vert f\Vert_{L^2([0,+\infty[;L^2(\Omega))}^2+\Vert\Delta u_0\Vert_{L^2(\Omega)}^2+\Vert u_1\Vert_{V(\Omega)}^2),
\end{equation}
where the constants $C>0$ depend only on $\Omega$ and more precisely of the constant in the Poincar\'e inequality on $V(\Omega)$.
\end{theorem}
An important corollary of this theorem is that it is possible to define a functional space of solutions of~(\ref{varformdampedwaveeqdirhom}) for the homogeneous initial data, which is isomorphic to space $L^2([0,+\infty[;L^2(\Omega))$ of the source terms.

\begin{theorem}\label{thmeqivcaudampwav}
For $\Omega$ a bounded Sobolev admissible domain in $\mathbb{R}^n$, let $-\Delta$ be the Laplacian operator associated with the mixed boundary conditions by~\eqref{EqVFPoisson}. Let
\begin{equation}\label{solspacedirhom}
X:= H^1(]0,+\infty[;\mathcal{D}_2(-\Delta))\cap H^2(]0,+\infty[;L^2(\Omega))
\end{equation}
be endowed with the norm (see Definition~\ref{defdomLapl} for $p=2$)
\begin{equation}
 \|u\|^2_X=\int_0^{+\infty}(\|\Delta u(t)\|^2_{L^2(\Omega)}+\|\del_t \Delta u(t)\|^2_{L^2(\Omega)}+\|\del_t u(t)\|^2_{V(\Omega)}+\|u(t)\|^2_{V(\Omega)}+\|\del_t^2u(t)\|^2_{L^2(\Omega)})\dt,
\end{equation}
and $X_0=\lbrace u\in X\vert u(0)=0,\;\partial_t u(0)=0\rbrace.$
Then there exists a unique weak solution $u\in X_0$ in the sense of formulation~(\ref{varformdampedwaveeqdirhom})
 of the boundary-valued problem~(\ref{dampedwaveeqmix}) with $u_0=u_1=0$
if and only if
$f\in L^2([0,+\infty[;L^2(\Omega))$.
Moreover the following estimate holds
$$\Vert u\Vert_{X}\leq C \Vert f\Vert_{L^2([0,+\infty[;L^2(\Omega))},$$
where $C>0$ depends only on $\Omega$.
\end{theorem}
\begin{proof}
As functions on time $t\mapsto u(t) \in H^1(]0,+\infty[)$ and $t\mapsto u'(t) \in H^1(]0,+\infty[)$ by the definition of $X$, we have, by Sobolev embedding theorem in one dimensional case, that for all $T>0$ $H^1(]0,T[)\subset C([0,T])$ and consequently for all $u\in X$ the trace values $u(0)$ and $\partial_t u(0)$ are well-defined.

 Theorem \ref{dampedwaveeqregint1} gives us directly one side of the equivalence. If $f\in L^2([0,+\infty[;L^2(\Omega))$, $u_0=0$ and $u_1=0$ there exists a unique $u$ weak solution of~(\ref{varformdampedwaveeqdirhom}) with $\partial^2_t u$, $\Delta u$ and $\Delta\partial_t u$ in $L^2([0,+\infty[;L^2(\Omega))$, along with the estimates~(\ref{aprioriglobdampedwave}) and~(\ref{aprioriglobdampedwavebis}) which implies $u\in X_0$ with the desired estimate.

 Now let us consider a weak solution $u\in X_0$ satisfying Definition~\ref{defweakdampwav}. By linearity $u$ is unique and, by regularity of $u$, from~(\ref{varformdampedwaveeqdirhom}) we have
$f\in L^2([0,+\infty[;L^2(\Omega)).$ 
\end{proof}
\subsection{$L^p$ regularity}
For $p \geq2$ we have $L^p(\Omega)\subset L^2(\Omega)$ and, by Theorem~\ref{thmeigenfuncLaprob} and by~\cite[Corollary~5.5]{Daners}, the spectrum of $-\Delta$ in $L^p(\Omega)$ is contained in $\mathbb{R}^*_+$. 
We give a result on maximal $L^p$ regularity, which is a direct application of Theorem~$4.1$ in Ref.~\cite{Srivastava} to the linear system for the strongly damped wave equation with mixed boundary conditions and homogeneous initial data~(\ref{dampedwaveeqmix}):
\begin{theorem}\label{thmmaxlpreg1}
Let $-\Delta$ be defined on $\mathcal{D}_p(-\Delta))$ as the $L^p$-realisation of the Laplacian for mixed boundary conditions by Definition~\ref{defdomLapl}. 
For $p\geq 2$ and $T>0$, there exits a unique weak solution $u\in X^p_0$ with 
\begin{equation}\label{solsparobadmdom}
X^p:=W^{1,p}([0,T];\mathcal{D}_p(-\Delta))\cap W^{2,p}([0,T];L^p(\Omega))
\end{equation}
endowed with the norm
\begin{equation}\label{EqNormXp}
 \|u\|^p_{X^p}=\int_0^{T}(\|\Delta u(t)\|^p_{L^p(\Omega)}+\|\del_t \Delta u(t)\|^p_{L^p(\Omega)}+\|\del_t u(t)\|^p_{L^p(\Omega)}+\|\del_t^2u(t)\|^p_{L^p(\Omega)})\dt,
\end{equation}
and
$$X^p_0:=\lbrace u\in X^p\vert u(0)=0,\;\partial_t u(0)=0\rbrace $$
 of the mixed boundary-valued problem~(\ref{dampedwaveeqmix}) with $u_0=u_1=0$
if and only if $f\in L^p([0,T];L^p(\Omega))$.
Moreover it holds the estimate
$$\Vert u\Vert_{X^p}\leq C \Vert f\Vert_{L^p([0,T];L^p(\Omega))}.$$
\end{theorem}
\begin{proof}
We observe that~\eqref{EqNormXp} is a norm thanks to the uniqueness of the weak solution of the Poisson mixed problem in the $L^p$ framework. In addition, let us notice that the time traces $u(0)$ and $\del_t u(0)$ are well-defined as, by their regularity in $X^p$, these functions are continuous on time.

The statement is a result of maximal $L^p$ regularity. It can be proved by different general methods~\cite[Section~4]{Batty} for abstract evolutive problems. It is also possible to use the abstract general result~\cite[Theorem~4.1]{Srivastava} involving theory of UMD spaces. By~\cite[Theorem~4.1]{Srivastava}, as $-\Delta$ is a sectorial operator on $L^p(\Omega)$ which admits a bounded $RH^{\infty}$ functional calculus of angle $\beta$ with $0<\beta<\frac{\pi}{2}$, then
system~(\ref{dampedwaveeqmix}) considered with $u_0=u_1=0$ has $L^p$-maximal regularity. 

UMD spaces have been introduced in~\cite{Bourgain}. By~\cite{Kalton}, if $A$ is a sectorial operator on an UMD space $X$ with property $(\alpha)$ and admits a bounded $H^{\infty}$ calculus of angle $\beta$, then $A$ admits a $RH^{\infty}$ calculus of angle $\beta$. For the definition of Banach spaces having property $(\alpha)$ see~\cite{Pisier}. For $p>1$, $L^p(\Omega)$ is an UMD space having property $(\alpha)$ according to~\cite[p.~752]{Srivastava}. 

 Thanks to~\cite[Thm.~5.6]{Arendt3}, the operator $-\Delta$ is a sectorial operator on $L^p(\Omega)$ which admits a bounded $H^{\infty}$ calculus of angle $\beta$ with $\beta<\frac{\pi}{2}$. The key point according to Theorem~\ref{thmeigenfuncLaprob}, holding on Sobolev admissible domains, is that we have for $z\in \mathbb{C}$ such that $\vert arg(z)\vert <\frac{\pi}{2}$
$$\Vert e^{z\Delta}\Vert_{L^2\rightarrow L^2}\leq e^{-\lambda_1 \vert z\vert} $$
with $\lambda_1>0$. The estimate in Theorem~\ref{thmmaxlpreg1} is a consequence of the closed graph theorem.
\end{proof}
Now we consider the non-homogeneous damped wave problem~(\ref{dampedwaveeqmix}):
\begin{theorem}\label{thmmaxlpreg2}
For $p\geq 2$ and $T>0$, let $X^p$ be defined by~(\ref{solsparobadmdom}). Moreover let us consider all elements of $L^p(\Omega)\times L^p(\Omega)$ which could be viewed as the traces at $t=0$ of an element of $X^p$ and of its time derivative:
\begin{equation}\label{tracesolsparobadmdom}
(L^p(\Omega),\mathcal{D}_p(-\Delta))_p=\lbrace (u_0,u_1)\in L^p(\Omega)\times L^p(\Omega)\vert \;\exists v\in X^p \hbox{ with }v(0)=u_0, v_t(0)=u_1\rbrace.
\end{equation}
Then there exits a unique weak solution $u\in X^p$
 of the damped wave equation problem~(\ref{dampedwaveeqmix})
if and only if $f\in L^p([0,T];L^p(\Omega))$ and $(u_0,u_1)\in (L^p(\Omega),\mathcal{D}_p(-\Delta))_p$.
Moreover we have the estimate
$$\Vert u\Vert_{X^p}\leq C (\Vert f\Vert_{L^p([0,T];L^p(\Omega))}+\Vert (u_0,u_1)\Vert_{(L^p(\Omega),\mathcal{D}_p(-\Delta))_p}).$$
\end{theorem}
\begin{proof}
For $(u_0,u_1)\in (L^p(\Omega),\mathcal{D}_p(-\Delta))_p$, we have by definition $w\in X^p$ such that 
$$w(0)=u_0\quad\hbox{and}\quad w_t(0)=u_1.$$
 Thus, the set $(L^p(\Omega),\mathcal{D}_p(-\Delta))_p$ gives us all situable $u_0$ and $u_1$ in $L^p(\Omega)$ for the initial data of~(\ref{dampedwaveeqmix}). 
 In particular, 
 $$\partial_t^2 w-c^2 \Delta w-\nu \Delta\partial_t w\in L^p([0,T];L^p(\Omega)) .$$ 
So in the sense of Theorem \ref{thmmaxlpreg1}, if we take $\tilde{w}$ the unique weak solution in $X^p$ of 
$$\left\lbrace
\begin{array}{l}
\partial_t^2 \tilde{w}-c^2 \Delta\tilde{w}-\nu \Delta\partial_t\tilde{w}=f-(\partial_t^2 w-c^2 \Delta w-\nu \Delta\partial_t w)\hbox{ on }[0,T]\times\Omega,\\
\frac{\partial }{\partial n}\tilde{w}+a \tilde{w}=0\hbox{ on }[0,T]\times\partial\Omega,\\
\tilde{w}(0)=\partial_t \tilde{w}(0)=0\hbox{ in }\Omega,
\end{array}
\right.
$$
we have by the linearity $u=w+\tilde{w}$ which is the weak solution of the damped wave equation problem~(\ref{dampedwaveeqmix}). The unicity comes from the unicity of the solution when $u_0=u_1=0$ by Theorem~\ref{thmmaxlpreg1}. The other side of the equivalence comes directly from the definition of $X^p$ and $(L^p(\Omega),\mathcal{D}_p(-\Delta))_p$.

The estimate is a consequence of the closed graph theorem. 
\end{proof}
\begin{remark}
Since 
$$\mathcal{D}_p(-\Delta)\times\mathcal{D}_p(-\Delta)\subset (L^p(\Omega),\mathcal{D}_p(-\Delta))_p$$
we have a similar estimate in Theorem \ref{thmmaxlpreg2} for the solutions of the damped wave equation problem~(\ref{dampedwaveeqmix}), when $(u_0,u_1)\in \mathcal{D}_p(-\Delta)\times\mathcal{D}_p(-\Delta)$ replacing $\Vert (u_0,u_1)\Vert_{(L^p(\Omega),\mathcal{D}_p(-\Delta))_p})$ by $\Vert u_0\Vert_{\mathcal{D}_p(-\Delta)}+\Vert u_1\Vert_{\mathcal{D}_p(-\Delta)}$.
\end{remark}
\section{Well-posedness of the Westervelt equation}\label{secwpWesdirmix}

In this section, $\Omega$ is a bounded Sobolev admissible domain in $\mathbb{R}^2$ or $\mathbb{R}^3$.

To be able to give a sharp estimate of the smallness of the initial data and at the same time to estimate the bound of the corresponding solution of the Westervelt equation~\eqref{EqWest}, we use the following theorem~\cite{Sukhinin}:
\begin{theorem}[Sukhinin]\label{thSuh}
 Let $X$ be a Banach space, let $Y$ be a separable
topological vector space, let $L : X \rightarrow Y$ be a linear
continuous operator, let $U$ be the open unit ball in $X$, let ${\rm
P}_{LU}:LX \to [0,\infty [$ be the Minkowski functional of the set
$LU$, and let $\Phi :X \to LX$ be a mapping satisfying the condition
\begin{equation*}
 {\rm P}_{LU} \bigl(\Phi (x) -\Phi (\bar{x})\bigr) \leq
\Theta (r) \left\|x -\bar{x} \right\|\quad \hbox{for} \quad \left\|x
-x_0 \right\| \leqslant r,\quad \left\|\bar{x} -x_0 \right\| \leq r
\end{equation*} for some $x_0 \in X,$ where $\Theta :[0,\infty [ \to [0,\infty [$ is a monotone
non-decreasing function. Set $b(r) =\max \bigl(1 -\Theta (r),0
\bigr)$ for $r \geq 0$.

 Suppose that $$w =\int\limits_0^\infty b(r)\,dr \in ]0,\infty ], \quad r_* =\sup \{ r
\geq 0|\;b(r) >0 \},$$

$$w(r) =\int\limits_0^r b(t)dt \quad (r \geq 0) \quad\hbox{and} \quad g(x) =Lx
+\Phi(x) \quad \hbox{for} \quad x \in X.$$
Then for any $r \in
[0,r_*[$ and $ y \in g(x_0) +w(r)LU$, there exists an
 $ x \in x_0 +rU$ such that $g(x) =y$.
\end{theorem}
\begin{remark}\label{remch22.1.} Theorem of Sukhinin allows treating the well-posedness questions of nonlinear problems once formulated in the operator abstract framework. We have a nonlinear operator $\Phi$, mapping from a Banach space $X$ to the image of the linear part, for which we know its local ``speed of growing'' $\Theta$. From this information, it is possible to define the maximum size of a neighborhood of a fixed point $x_0$ in which there exists a solution $x$ if the source term $y$ belongs to the maximal possible set $g(x_0) +w(r)LU$. In other words, this theorem allows obtaining sharp estimates of the size of the source term and the corresponding size of the solution. 
If either $L$ is injective or $KerL$ has a topological
complement $E$ in $X$ such that $L(E \cap U) =LU$, then the
assertion of Theorem~\ref{thSuh} follows from the contraction
mapping principle~\cite{Sukhinin}. In particular, if $L$ is injective,
then the solution is unique. For additional examples of the applications of this theorem in solving direct nonlinear problems, there are Refs.~\cite{Perso,dekthese} and for inverse problems context see~\cite{ROZANOVA-2004,ROZANOVA-2004-1,ROZANOVA-2005}.
\end{remark}

With the help of Theorem~\ref{thSuh} we prove the following global well-posedness result.
\begin{theorem}\label{ThWPWestGlob}
Let $\Omega$ be a bounded Sobolev admissible domain in $\mathbb{R}^2$ or $\mathbb{R}^3$.
 Assume $\nu>0$ and $p\geq 2$. Let $X^p$ be the space defined in~(\ref{solsparobadmdom}) with $T>0$ if $p>2$ and $T=+\infty$ if $p=2$. Suppose
 \begin{align*}
 u_0\in \mathcal{D}_p(-\Delta)\hbox{ in }L^p,\quad & u_1\in \mathcal{D}_p(-\Delta)\hbox{ in }L^p\hbox{ if }p>2\hbox{, or } u_1\in V(\Omega)\hbox{ if }p=2\\
 &\hbox{and }f\in L^p([0,+\infty[;L^p(\Omega)),
 \end{align*}
and in addition, let $C_1$ be the minimal constant for which the weak
 solution, in the sense of~(\ref{varformdampedwaveeqdirhom}), $u^*\in X^p$ of the corresponding non homogeneous linear boundary-valued problem~(\ref{dampedwaveeqmix}) satisfies if $p>2$ 
 $$\|u^*\|_{X^p}\leq C_1 (\Vert f\Vert_{L^p([0,+\infty[;L^p(\Omega))} +\Vert u_0\Vert_{\mathcal{D}_p(-\Delta)}+\Vert u_1\Vert_{\mathcal{D}_p(-\Delta)})$$
 and if $p=2$ satisfies (in this case $C_1=\frac{C_2}{\nu}$ where $C_2$ only depends on $\Omega$ by the constant in the Poincar\'e inequality on $V(\Omega)$)
 $$\|u^*\|_{X^2}\le \frac{C_1}{\nu }(\Vert f\Vert_{L^2([0,+\infty[;L^2(\Omega))} +\Vert u_0\Vert_{\mathcal{D}_2(-\Delta)}+\Vert u_1\Vert_{V(\Omega)}).$$
 
Then there exists $r_{*}>0$ such that for all $r\in[0,r_{*}[$
 and all 
 data satisfying if $p>2$
 $$ \Vert f\Vert_{L^p([0,+\infty[;L^p(\Omega))} +\Vert u_0\Vert_{\mathcal{D}_p(-\Delta)}+\Vert u_1\Vert_{\mathcal{D}_p(-\Delta)})\leq \frac{1}{C_1}r$$
 and if $p=2$
$$
 \Vert f\Vert_{L^2([0,+\infty[;L^2(\Omega))} +\Vert u_0\Vert_{\mathcal{D}_2(-\Delta)}+\Vert u_1\Vert_{V(\Omega)}\leq \frac{\nu}{C_2}r,
$$
there exists the unique weak solution $u\in X^p$ of the mixed boundary valued problem for the Westervelt equation
\begin{equation}\label{CauchypbWesdirhom}
\left\lbrace
\begin{array}{l}
\partial^2_t u-c^2\Delta u-\nu \Delta\partial_t u=\alpha u \partial^2_t u+\alpha (\partial_t u)^2+f\quad on\quad [0,T]\times\Omega,\\
u=0 \hbox{ on }\Gamma_D\times[0,T],\\
\frac{\partial}{\partial n}u=0\hbox{ on }\Gamma_N\times[0,T],\\
\frac{\partial}{\partial n}u+a u=0\hbox{ on }\Gamma_R\times[0,T],\\
u(0)=u_0, \hbox{ }\partial_t u(0)=u_1.
\end{array}\right.
\end{equation} 
in the following sense: for all $\phi\in L^2([0,T];V(\Omega))$ 
\begin{align}
\int_0^{T}&(\partial^2_t u,\phi)_{L^2(\Omega)}+c^2( u, \phi)_{V(\Omega)}+\nu ( \partial_t u, \phi)_{V(\Omega)}ds\nonumber\\
&=\int_0^{T} \alpha (u \partial^2_tu+ (\partial_t u)^2+f,\phi)_{L^2(\Omega)}ds,\label{weakformWes}
\end{align}
with $u(0)=u_0$ and $\partial_t u(0)=u_1$.
Moreover 
$$\Vert u\Vert_{X^p}\leq 2 r.$$ 
 
\end{theorem}
\begin{proof}
For $p=2$, $T=+\infty$ $u_0\in \mathcal{D}_p(-\Delta)$ and $u_1\in V(\Omega)$ and $f\in L^2([0,+\infty[;L^2(\Omega))$ let us denote by $u^*\in X^2$ the unique weak solution, existing by Theorem~\ref{dampedwaveeqregint1}, of the linear problem~(\ref{dampedwaveeqmix})
in the sense of the variational formulation~(\ref{varformdampedwaveeqdirhom}).

 According to Theorem~\ref{thmeqivcaudampwav} , $X^2=X$ defined in~(\ref{solspacedirhom}), hence we denote $X_0^2:=X_0$ and in addition take $Y=L^2[0,+\infty[;L^2(\Omega))$. 
Then by Theorem~\ref{thmeqivcaudampwav}, the linear operator 
$$L:X_0^2\rightarrow Y,\quad u\in X_0^2\mapsto\;L(u):=u_{tt}-c^2\Delta u-\nu \Delta u_t\in Y,$$
 is a bi-continuous isomorphism. 
 
 Let us now notice that if $v$ is the unique weak solution of the non-linear mixed boundary valued problem 
 \begin{equation}\label{SystwesnV}
\begin{cases}
v_{tt}-c^2\Delta v-\nu \Delta v_t-\alpha (v+u^*)(v+u^*)_{tt}-\alpha [(v+u^*)_t]^2=0 \hbox{ on } [0,+\infty[\times\Omega,\\
v=0 \hbox{ on }\Gamma_D\times[0,+\infty[,\quad
\frac{\partial}{\partial n}v=0\hbox{ on }\Gamma_N\times[0,+\infty[,\\
\frac{\partial}{\partial n}v+a v=0\hbox{ on }\Gamma_R\times[0,+\infty[,\\
v(0)=0,\quad v_t(0)=0,
\end{cases}
\end{equation}
 then $u=v+u^*$ is the unique weak solution of the boundary valued problem for the Westervelt equation~(\ref{CauchypbWesdirhom}).
 Let us prove using Theorem~\ref{thSuh} the existence of a such $v\in X_0^2$, which is the unique weak solution of~(\ref{SystwesnV}) in the following sense: for all $\phi\in L^2([0,+\infty[;V(\Omega))$
 \begin{align*}
\int_0^{+\infty}&(\partial^2_t v,\phi)_{L^2(\Omega)}+c^2( v, \phi)_{V(\Omega)}+\nu ( \partial_t v, \phi)_{V(\Omega)}ds\nonumber\\
&=\int_0^{+\infty} \alpha ((v+u^*)(v+u^*)_{tt}+[(v+u^*)_t]^2,\phi)_{L^2(\Omega)}ds
 \end{align*}
 with $v(0)=0$ and $\partial_t v(0)=0$.
 
We suppose that $\Vert u^*\Vert_{X^2}\leq r$
and define for $v\in X_0^2$
$$\Phi(v):=\alpha (v+u^*)(v+u^*)_{tt}+\alpha [(v+u^*)_t]^2.$$

For $w$ and $z$ in $X_0^2$ satisfying
$$\Vert w\Vert_{X^2}\leq r \quad \hbox{and} \quad \Vert z\Vert_{X^2}\leq r,$$
 we estimate
$
 \Vert \Phi(w)-\Phi(z)\Vert_Y
$
by applying the triangular inequality.
The key point is that it appears terms of the form $\Vert g b_{tt}\Vert_Y$ and $\Vert g_t b_t\Vert_Y$
with $g$ and $b$ in $X^2$ and we have the estimate
\begin{align*}
\Vert g b_{tt}\Vert_Y\leq &\Vert g\Vert_{L^{\infty}(\mathbb{R}^+\times\Omega)}\Vert b_{tt}\Vert_Y.
\end{align*}
By Theorem~\ref{thmlinfestL2} which ensures for elements of $\mathcal{D}_p(-\Delta)$ the inequality $\|g\|_{L^\infty(\Omega)}\le C\|\Delta g\|_{L^p(\Omega)}$, we have
\begin{align*}
\Vert g b_{tt}\Vert_Y\leq & C \Vert g\Vert_{L^{\infty}(\mathbb{R}^+;\mathcal{D}_p(-\Delta))}\Vert b\Vert_{X^2}\\
\intertext{and the Sobolev embedding implies}
\Vert g b_{tt}\Vert_Y\leq & C \Vert g\Vert_{H^1(\mathbb{R}^+;\mathcal{D}_p(-\Delta))}\Vert b\Vert_{X^2}
\\
\leq & B_1 \Vert g\Vert_{X^2} \Vert b\Vert_{X^2},
\end{align*}
with $B_1$ depending only on $\Omega$.
Moreover, we have 
\begin{align*}
\Vert g_t b_t\Vert_Y\leq & \sqrt{\int_{0}^{+\infty} \Vert g_t\Vert_{L^{\infty}(\Omega)} \Vert b_t\Vert_{L^2(\Omega)}ds}.
\intertext{Therefore, again by Theorem~\ref{thmlinfestL2} we find}
\Vert g_t b_t\Vert_Y\leq & C \sqrt{\int_{0}^{+\infty} \Vert g_t\Vert_{\mathcal{D}_p(-\Delta)} \Vert b_t\Vert_{L^2(\Omega)}ds}\\
\leq & C \Vert g_t\Vert_{L^{2}([0,+\infty[; \mathcal{D}_p(-\Delta))}\Vert b_t\Vert_{L^{\infty}([0,+\infty[; L^2(\Omega))}\\
\leq & C \Vert g\Vert_{X^2} \Vert b_t\Vert_{H^1([0,+\infty[; L^2(\Omega))}
\end{align*}
also using Sobolev's embeddings. Finally it holds 
$$\Vert g_t b_t\Vert_Y\leq B_2 \Vert g\Vert_{X^2} \Vert b\Vert_{X^2}$$
with $B_2$ depending only on $\Omega$.
Taking $g$ and $b$ equal to $u^*$, $w$, $z$ or $w-z$, and supposing that $\Vert u^*\Vert_{X^2}\leq r$, $\Vert w\Vert_{X^2}\leq r$ and $\Vert z\Vert_X\leq r$, we obtain
\begin{align*}
\Vert \Phi(w)-\Phi(z)\Vert_Y\leq 8
\alpha B r \Vert w-z\Vert_{X^2}
\end{align*}
with $B=\max(B_1,B_2)>0$ depending only on $\Omega$.

By the fact that $L$ is a bi-continuous isomorphism, there exists a minimal constant $C_{\nu}=C\left(\frac{1}{ \nu} \right)>0$ (coming from the inequality $\|u\|_X\leq C \|f\|_Y$ for $u$, a solution of the linear problem~(\ref{dampedwaveeqmix}) with homogeneous initial data) %
such that
$$\forall u\in X_0^2 \quad \Vert u\Vert_{X^2}\leq C_{\nu} \Vert Lu\Vert_Y.$$
Hence, for all $g\in Y$
$$P_{LU_{X_0^2}}(g)\leq C_{\nu} P_{U_Y}(g)=C_{\nu}\Vert g\Vert_Y.$$
Then we find for $w$ and $z$ in $X_0^2$, such that $\|w\|_{X^2}\le r$, $\|z\|_{X^2}\le r$, and also with $\|u^*\|_{X^2}\le r$, that
$$P_{LU_{X_0^2}}(\Phi(w)-\Phi(z))\leq \Theta(r) \Vert w-z\Vert_X,$$
where $\Theta(r):=8 B C_{\nu} \alpha r$.
Thus we apply Theorem~\ref{thSuh} for 
\\$g(x)=L(x)-\Phi(x)$ and $x_0=0$. Therefore, knowing that $C_\nu=\frac{C_0}{ \nu}$, we have, that for all $r\in[0,r_{*}[$ with
\begin{equation}\label{Eqret}
 r_{*}=\frac{1}{8 B C_{\nu} \alpha},
\end{equation}
 for all $y\in \Phi(0)+w(r) L U_{X_0^2}\subset Y$
with $$w(r)= r-4 B C_{\nu} \alpha r^2,$$
there exists a unique $v\in 0+r U_{X_0^2}$ such that $L(v)-\Phi(v)=y$.
But, if we want that $v$ be the solution of the non-linear Cauchy problem~(\ref{SystwesnV}), then we need to impose $y=0$, and thus to ensure that $0\in \Phi(0)+w(r) L U_{X_0^2}$.
Since $-\frac{1}{w(r)}\Phi(0)$ is an element of $Y$ and $LX_0^2=Y$, there exists a unique $z\in X_0$ such that
\begin{equation}\label{Eqz}
 L z=-\frac{1}{w(r)}\Phi(0).
\end{equation}
Let us show that $\|z\|_X^2\le 1$, what will implies that $0\in \Phi(0)+w(r) L U_{X_0}$.
Noticing that
\begin{align*}
\Vert \Phi(0)\Vert_Y & \leq \alpha \Vert u^*_t u^*_{tt}\Vert_Y +\alpha \Vert u^*_t u^*_t\Vert_Y\\
& \leq 2 \alpha B \Vert u^*\Vert_{X^2}^2 \leq 2 \alpha B r^2
\end{align*}
and using~(\ref{Eqz}), we find
\begin{align*}
 \Vert z\Vert_{X^2}&\leq C_{\nu}\Vert L z\Vert_Y=C_\nu\frac{\Vert \Phi(0)\Vert_Y}{w(r)}\\
 &\leq \frac{C_{\nu} 2 B \alpha r}{(1-4 C_{\nu} B \alpha r)}<\frac{1}{2},
\end{align*}
as soon as $r<r^*$.

Consequently, $z\in U_{X_0^2}$ and $\Phi(0)+w(r) Lz=0$.

Then we conclude that for all $r\in[0,r_{*}[$, if $\|u^*\|_{^2}\le r$, there exists a unique $v\in r U_{X_0}$ such that $L(v)-\Phi(v)=0$, $i.e.$ the solution of the non-linear Cauchy problem~(\ref{SystwesnV}).
Thanks to the maximal regularity and a priori estimate following from Theorem~\ref{thmeqivcaudampwav},
there exists a constant $C_1=C_1(\Omega)$, such that
$$\|u^*\|_{X^2}\le \frac{C_1}{\nu }(\Vert f\Vert_Y+\Vert u_0\Vert_{\mathcal{D}_p(-\Delta)}+\Vert u_1\Vert_{V(\Omega)}).$$

Thus, for all $r\in[0,r_{*}[$ and $\Vert f\Vert_Y+\Vert u_0\Vert_{\mathcal{D}_p(-\Delta)}+\Vert u_1\Vert_{V(\Omega)} \leq \frac{\nu }{C_1}r$, the function $u=u^*+v\in X$ is the unique solution of the Cauchy problem for the Kuznetsov equation and $\Vert u\Vert_{X^2}\leq 2 r$.

Let us notice that when $f=0$ we have
$$\Vert u^*\Vert_{X^2}\leq \frac{C_1'}{\sqrt{\nu }}(\Vert u_0\Vert_{\mathcal{D}_p(-\Delta)}+\Vert u_1\Vert_{V(\Omega)}).$$
The case $p>2$ and $0<T<+\infty$ is essentially the same
and thus is omitted. We just replace $L^2([0,+\infty[;L^2(\Omega))$ by $L^p([0,T];L^p(\Omega))$. We also use the Theorems~\ref{thmmaxlpreg2} and~\ref{thmlinfestL2} to have the required estimates following from the fact that for $p>2$ $$W^{1,p}([0,+\infty[)\subset L^{\infty}([0,+\infty[).$$ 
\end{proof}

\section{Approximation of the fractal problem for the Westervelt equation by prefractal problems with Lipschitz boundaries}
\label{secconvprefrfr}
\subsection{Uniform domains in $\mathbb{R}^n$ with self-similar boundaries and their polyhedral approximations}\label{subsecOSC}

In this section, we give conditions under which a domain $\Omega$ in $\mathbb{R}^n$ with piece-wise self-similar boundary is a uniform domain. Moreover, under our conditions, these domains have natural polyhedral approximations $\Omega_m$ which are uniformly $(\epsilon,\infty)$-domains, that is, $\epsilon$ does not depend on $m$. Our conditions cover the 
examples of scale-irregular Koch curves~\cite{Capitanelli,Capitanelli2}, 
the square Koch curve, also called the Minkowski fractal~\cite{SapovalGobron,PhysRevLett.83.726,PhysRevLett.108.240602},
and their $n$-dimensional analogs. We do not give the most general assumptions but rather concentrate on the situations with potential practical applications, such as~\cite{MAGOULES-2021}. 

Suppose $\Omega_0$ is a polyhedron in $\mathbb{R}^n$ and $K_0$ is one of its faces. 
We denote the $(n-2)$-dimensional hypersurface boundary of $K_0$ by $\partial_{(n-2)}K_0$, which is just the union of $n-2$-dimensional faces of $K_0$. 
A typical example is $\Omega_0=[0,1]^n$ is the unite hypercube in $\mathbb{R}^n$ and $K_0=[0,1]^{n-1}\times\{0\}$. In this case $\partial_{(n-2)}K_0$ is the 
$(n-2)$-dimensional hypersurface boundary of the $(n-1)$-dimensional hypercube 
$K_0=[0,1]^{n-1}\times\{0\}$. 

We suppose that polyhedral hypersurfaces $K_m$ are defined inductively using a sequence of iterating function systems of $N_m$ contractive similitudes 
$$(\psi_{i,m})_{1\leq i\leq N_m}$$
 with contraction factors 
 $$(d_{i,m})_{1\leq i\leq N_m}$$ 
 by 
 $$K_m= \Phi_m(K_0):=\cup_{i=1}^{N_m}\Psi_{i,m}(K_0).$$
 These are standard concepts, 
 which we do not discuss in our paper in detail, 
 are thoroughly described, for instance, in~\cite{Falconer,ALM2003}
 (see also Appendix~\ref{subsecexKoch}). 
 
 Throughout the rest of the paper, we make the following assumptions which allow approximating fractal domains with polygonal hypersurfaces. 
 
 \begin{assumption}
 [Fractal Self-Similar Face]\label{a-face}
 We assume that 
each $K_m$ is a polygonal surface with $(n-2)$-dimensional hypersurface boundary that is the same 
 as the $(n-2)$-dimensional hypersurface boundary of $K_0$.

\end{assumption} 
\begin{assumption}[Uniform sequence of uniform domains]\label{a-uniuni}
 We assumption that $\Omega_m$ and $\Omega$ are 
 uniformly 
 exterior and interior 
 $(\epsilon,\infty)$-domains, that is, $\epsilon$ does not depend on $m$.
\end{assumption} 


We also assume the   standard  open set condition \cite[Section 9.2]{Falconer}, which means that there is a non-empty bounded open set $O$ such that $\Phi_m(O)\subset O$ with the union in the left-hand side disjoint.

\subsection{Trace and extension theorems in the approximation framework of self similar fractals}
\label{subsectrace}

In this subsection we assume the same notation and assumptions as in Subsection~\ref{subsecOSC}. For $N$ contraction factors $d_i$, $i=1, \ldots, N$ we define
 \begin{equation}\label{bigcontrfact}
 D=\sum_{i=1}^N d_i^{n-1}.
 \end{equation}
 With notations $w|m=(w_1,\ldots,w_n)$ for $w_i\in\{1,\ldots N\}$ and $$\psi_{w|m}=\psi_{w_1}\circ\ldots \circ \psi_{w_m},$$ we introduce the volume measure $\mu$ as the unique Radon measure on $$K=\overline{\bigcup_{m=1}^{+\infty}\Phi_m(K_0)},$$ such that
\begin{equation}\label{radmesprop}
\mu(\psi_{i\vert m}(K))=\frac{\prod_{i=1}^m d_{w_{i}}^{n-1}}{D^m}.
\end{equation}

While the fractal boundary $\partial \Omega$ is irregular, the prefractal boundary $\partial \Omega_m$ is polygonal, thus Lipschitz. Hence, we can easily give well-posedness results for partial differential equations with domains having polygonal boundary using the classic Lebesgue measure on $\partial \Omega_m$. Then we can obtain a well-posedness result for the solution $u$ of the Westervelt equation on an irregular domain $\Omega$ by a convergence argument on the functions $u_m$, solutions of the Westervelt equation on domains $\Omega_m$. This approach also allows constructing an approximation of $u$.
In order to do so, the following results are needed. 

\begin{theorem}\label{convtrace}
Under Assumptions~\ref{a-face} and \ref{a-uniuni}, let
\begin{equation}\label{defsigmn}
\sigma_m:=\frac{1}{D^m},
\end{equation}
where $D$ defined by~(\ref{bigcontrfact}).
 For any function $g\in H^1(\mathbb{R}^n)$ 
\begin{equation}
\sigma_m\int_{K_{m} } Tr_{K_m} g ds\rightarrow \int_{K}Tr_{K} g \;d\mu \hbox{ for }m\rightarrow +\infty.
\end{equation}
\end{theorem}
\begin{proof}
Let firstly suppose that $g\in C(\mathbb{R}^n)$. We follow the proof of~Theorem~2.1 in Ref.~\cite{Capitanelli2} given for the particular case of von Koch snowflake. 
For a fixed summit $A$ on $\partial K_0$ we introduce the measure
$$\mu_m=\sum_{(w_1,\ldots,w_n)\in \lbrace 1,\ldots,N\rbrace^m} \frac{\prod_{i=1}^m d_{w_i}^{n-1}}{D^m} \delta_{\psi_{w\vert m}(A)}.$$
Let us prove that $\mu_m$ weakly converges to the measure $\mu$ considered on $K$.
For any $m$, we introduce the following positive linear functional on the space $C(K)$
$$G_m(h)= \sum_{(w_1,\ldots,w_m)\in \lbrace 1,\ldots,N\rbrace^m} h(\psi_{w\vert m}(A)).$$
As $K$ is compact of $\R^n$, then $h\in C(K)$ is uniformly continuous on $K$. Consequently, we have
$$\forall \varepsilon>0, \quad \exists q\in \N^* \hbox{ such that } \forall n,m> q \quad \vert G_n(h)-G_m(h)\vert<\varepsilon.$$ 
 Thus for each fixed $h\in C(K)$ the numerical sequence $(G_m(h))_{m\in \N^*}$ converges. 
 Hence the limit defines a positive linear functional on $C(K)$. By the Riesz representation theorem, there exists a unique (positive) Borel measure $\tilde{\mu}$ such that
$$\lim_{m\rightarrow \infty} G_n(h)=\int_{K}h \;d\tilde{\mu}.$$
Moreover $\tilde{\mu}$ satisfies~(\ref{radmesprop}). Hence, from the uniqueness of a such measure, $\mu$ and $\tilde{\mu}$ coincide, and we obtain 
\begin{equation}\label{weakconvmum}
\forall h\in C(K) \quad \lim_{m\rightarrow\infty}\int_{K} h \;d\mu_m=\int_{K} h{\rm d} \mu ,
\end{equation}
which is the definition of the weak convergence of $\mu_m$ to $\mu$. We also notice that 
$\mu(K)=1$.

Let us formally write
\begin{multline}\label{weakconvboundpreffrac}
\left\vert \frac{1}{D^m}\int_{K_m} g\;ds-\int_{K} g\;d\mu\right\vert\\
\leq \left\vert \frac{1}{D^m}\int_{K_m} g\;ds-\int_K g\;d\mu_m\right\vert+\left\vert \int_K g\;d\mu_m-\int_{K} g\;d\mu\right\vert.
\end{multline}
Since
\begin{multline*}
 \int_{K_m} g\ds= \sum_{(w_1,\ldots,w_m)\in \lbrace 1,\ldots,N\rbrace^m} \int_{\psi_{w\vert n}} g ds\\
 = \sum_{(w_1,\ldots,w_m)\in \lbrace 1,\ldots,N\rbrace^m} \frac{\prod_{i=1}^m d_{w_i}^{n-1}}{D^m} g(\psi_{w\vert m}(P_{w\vert n})) ,
\end{multline*}
where $P_{w\vert m}\in K_0$ and under the assumption the Lebesgue measure $\lambda(K_0)=1$, the first term on the right hand side of~(\ref{weakconvboundpreffrac}) can be estimated by using the uniform continuity of $g$ as
\begin{multline*}
 \left\vert \frac{1}{D^m}\int_{K_m} g\;ds-\int_K g\;d\mu_m\right\vert \leq \\\sum_{(w_1,\ldots,w_m)\in \lbrace 1,\ldots,N\rbrace^m} \vert g(\psi_{w\vert m}(P_{w\vert n}))-g(\psi_{w\vert m}(A)\vert \frac{\prod_{i=1}^m d_{w_i}^{n-1}}{D^m}.
\end{multline*}

As the second term on the right-hand side of~(\ref{weakconvboundpreffrac}) can be estimated by using~(\ref{weakconvmum}) we achieve the desired result for $g\in C(\mathbb{R}^n)$. To obtain the same for $g\in H^1(\mathbb{R}^n)$ we apply the density argument and~\cite[Thm.~3.5]{Capitanelli}.
\end{proof}
To be able to control the traces on the prefractal boundaries we generalize Lemma~3.1~\cite{Capitanelli2} for our $n$-dimensional case.
\begin{lemma}\label{lemmeasbound2}
Let $K_m$ be the $m$-th prefractal set. Then
$$\forall P\in \mathbb{R}^n \quad \mathcal{H}^{n-1}(B(P,r)\cap K_m)\leq C D^m r,$$
where the constant $C>0$ is independent on $m$, $B(P,r)$ denotes the Euclidean ball with center in $P$ and radius $0<r\leq 1$ and $\mathcal{H}^{n-1}$ is the $(n-1)-$dimensional Lebesgue measure.
\end{lemma}
\begin{proof}
Let us fixe $h\in \N$ such that
$$(\max d_i^{n-1})^h<r\leq (\max d_i^{n-1})^{h-1}.$$
Then $B(P,r)\subset B(P, (\max d_i^{n-1})^{h-1}).$ 

When $h>m$, since $\max d_i>\frac{1}{N}$, it holds
\begin{align*}
\mathcal{H}^{n-1}(B(P,r)\cap K_m)\leq & \mathcal{H}^{n-1}(B(P, (\max d_i^{n-1})^{h-1})\cap K_m)\\
\leq & C_1 (\max d_i^{n-1})^{h-1}<C_1 N^{n-1} r,
\end{align*}
where $C_1$ is independent of $m$. 

Let us now consider the case when $h\leq m$. 
Recall that the open set condition \cite[Section 9.2]{Falconer} means that $\Phi_m(O)\subset O$ with the union in the left-hand side disjoint. 
There are at most $C_2$ open sets $\psi_{w\vert h-1}(O)=\psi_{w_1}\circ\cdots\circ\psi_{w_n}$ ($C_2$ independent of $m$), where $O$ is the set of the open set condition associated to $(\psi_i)_{1\leq i\leq N}$, that has not empty intersection with $B(P, (\max d_i^{n-1})^{h-1})$. Then as 
\begin{align*}
\mathcal{H}^{n-1}( B(P, (\max d_i^{n-1})^{h-1})\cap K_m\cap \psi_{w\vert h-1}(O))\leq & D^m \frac{\prod_{i=1}^{h-1} d_{w_i}^{n-1}}{D^{h-1}}\\
\leq & D^m \frac{(\max d_i^{n-1})^{h-1}}{D^{h-1}},
\end{align*}
we obtain
\begin{align*}
\mathcal{H}^{n-1}(B(P,r)\cap K_m)\leq & \mathcal{H}^{n-1}(B(P, (\max d_i^{n-1})^{h-1})\cap K_m)\\
\leq & C_2 D^m \frac{(\max d_i^{n-1})^{h-1}}{D^{h-1}}\leq N C_2 D^m r.
\end{align*}
\end{proof}
Therefore we have the following uniform trace estimate for the prefractal boundaries with an analogous estimate in the fractal case:
\begin{theorem}\label{THcontrl2tr}
Let $u\in H^{\sigma}(\mathbb{R}^n)$ and $\frac{1}{2}<\sigma\leq 1$. Then for all $m\in \N$
\begin{equation}\label{contrl2trpref}
\frac{1}{D^m}\Vert Tr_{K_m} u\Vert^2_{L^2( K_m)}\leq C_{\sigma} \Vert u\Vert^2_{H^{\sigma}(\mathbb{R}^n)},
\end{equation}
where $C_{\sigma}>0$ is a constant independent of $m$. In addition, on the fractal $K$ with the measure $\mu$ satisfying~\eqref{EqMuUP} it also holds for $\frac{n-d}{2}<\sigma \le \frac{n}{2}$, $0<d\le n$ and for a constant $C_{\sigma}>0$
\begin{equation}\label{contrl2trfrac}
\Vert Tr_{K} u\Vert^2_{L^2(K)}\leq C_{\sigma} \Vert u\Vert^2_{H^{\sigma}(\mathbb{R}^n)}.
\end{equation} 
\end{theorem}
\begin{proof}
The proof of~(\ref{contrl2trpref}) is essentially the same as for~(\ref{contrl2trfrac}) proved in Ref.~\cite{Capitanelli2} and is thus omitted, the key point being Lemma~\ref{lemmeasbound2} and the use of Bessel kernels with Lemma~1 on p.~104 in Ref.~\cite{JONSSON-1984}. In addition~(\ref{contrl2trfrac}) is a direct consequence of~\cite[Theorem~5.1]{HINZ-2021}. 
\end{proof}
The following theorem extends functions of $H^1(\Omega_m)$ to space $H^1(\mathbb{R}^n)$ by an operator whose norm is independent of the (increasing) number of sides. It is a particular case of the extension theorem due to Jones (Theorem 1 in Ref.~\cite{JONES-1981}) as the domains $\Omega_m$ are $(\varepsilon,\infty)-$domains with $\varepsilon$ independent of $m$. We also give the extension result for the limit domain $\Omega$ coming from the Rogers extension theorem~\cite{Rogers} due to a ``degree-independent'' operator for Sobolev spaces on $(\eps,\infty)$-domains.
\begin{theorem}\label{extR2prefract}
For any $m\in \mathbb{N}$, there exists a bounded linear extension operator 
\\$E_{\Omega_m}:H^1(\Omega_m)\rightarrow H^1(\mathbb{R}^n)$, whose norm is independent of $m$, that is
\begin{equation}
\Vert E_{\Omega_m} v\Vert_{H^1(\mathbb{R}^n)}\leq C_J \Vert v\Vert_{H^1(\Omega^{n})}
\end{equation}
with a constant $C_J>0$ independent of $m$.

In addition, for the $(\eps,\infty)$-domain $\Omega$ with a fractal boundary $K$ there exists a bounded linear extension operator $E_{\Omega}:H^{\sigma}(\Omega)\rightarrow H^{\sigma}(\mathbb{R}^n)$, $\frac{1}{2}<\sigma\leq 1$, such that
\begin{equation}\label{extR2fract}
\Vert E_{\Omega} v\Vert_{H^{\sigma}(\mathbb{R}^n)}\leq C_{\Omega} \Vert v\Vert_{H^{\sigma}(\Omega)}.
\end{equation}
\end{theorem}
\begin{proof}
The independence on $m$ comes from the fact that the $\Omega_m$ are $(\varepsilon,\infty)$ domains with $\varepsilon$ fixed according to Assumption~\ref{a-uniuni}.
Then we just have to apply the result of Ref.~\cite{JONES-1981} on quasiconformal mappings.
The extension result for the fractal domain $\Omega$ follows from the Rogers extension theorem~\cite[Thm.~8]{Rogers}, since by its definition $\Omega$ is $(\eps,\infty)$-domain, with the use of interpolation techniques (see also~\cite[Thm.~3.5]{Capitanelli2}). 
\end{proof}

\subsection{Mosco type convergence}\label{subsecMosco}
We consider a domain $\Omega$ of $\R^n$ defined in Subsection~\ref{subsecOSC} and its polyhedral approximation by domains $\Omega_m$. We suppose as in Section~\ref{secwpdampwavmix} that 
$$\partial\Omega=\Gamma_{D,\Omega}\cup\Gamma_{N,\Omega}\cup \Gamma_{R,\Omega} \quad \hbox{with} \quad
 \Gamma_{R,\Omega}=K $$
and
$$\partial\Omega_m=\Gamma_{D,\Omega_m}\cup\Gamma_{N,\Omega_m}\cup \Gamma_{R,\Omega_m} \quad \hbox{with} \quad \Gamma_{R,\Omega_m}=K_m, $$
where the parts of boundaries with letters $D$, $N$ and $R$ correspond to the type of the homogeneous boundary condition considered on them: the Dirichlet, the Neumann and the Robin boundary conditions respectively.

Our aim is to consider the limit $m\to +\infty$ of the weak solutions of the following Westervelt mixed boundary problem to show that a weak solution on $\Omega$ can be approximated by weak solutions on $\Omega_m$.
\begin{equation}\label{Kuzeqrobinpref}
\left\lbrace
\begin{array}{l}
\partial^2_t u-c^2 \Delta u-\nu \Delta\partial_t u=\alpha\partial_t[u\partial_t u]+f\;\;\hbox{on}\;]0,+\infty[\times \Omega_m,\\
u=0 \hbox{ on }\Gamma_{D,\Omega_m}\times[0,+\infty[,\\
\frac{\partial}{\partial n}u=0\hbox{ on }\Gamma_{N,\Omega_m}\times[0,+\infty[,\\
\frac{\partial}{\partial n}u+a_m u=0\hbox{ on }K_m\times[0,+\infty[,\\
u(0)=u_{0,m},\;\;u_t(0)=u_{1,m} \;\;\hbox{on}\;\Omega_m.
\end{array}\right.
\end{equation}

We set 
\begin{align}
H(\Omega):=H^1([0,+\infty[;H^1(\Omega))\cap & H^2([0,+\infty[;L^2(\Omega))\label{Moscospace}
\end{align}
and fixe a Sobolev admissible domain $\Omega^*$ such that $\Omega\subset\Omega^*$ for all $m\in\N^*$ $\Omega_m\subset\Omega^*$.

For $u\in H(\Omega^*)$ and $\phi\in L^2([0,+\infty[,H^1(\Omega^*))$ we define
\begin{align}
F_m[u,\phi]:=&\int_0^{+\infty}\int_{\Omega_m} \partial_t^2 u \phi+c^2 \nabla u \nabla \phi+\nu \nabla \partial_t u \nabla \phi\;d\lambda\dt\nonumber\\
&+\int_0^{+\infty} \int_{K_m} c^2 a_m Tr_{\partial\Omega_m} u \;Tr_{\partial\Omega_m} \phi + \nu a_m Tr_{\partial\Omega_m}\partial_t u \;Tr_{\partial\Omega_m} \phi\ds\dt\label{Fnu}\\
&\int_0^{+\infty}\int_{\Omega_m} -\alpha (u \partial_t^2 u) \phi -\alpha ( \partial_t u)^2 \phi+f\phi\;d\lambda\dt\nonumber
\end{align}
and also
\begin{align}
F[u,\phi]:=&\int_0^{+\infty}\int_{\Omega} \partial_t^2 u \phi+c^2 \nabla u \nabla \phi+\nu \nabla \partial_t u \nabla \phi\;d\lambda\dt\nonumber\\
&+\int_0^{+\infty} \int_{K} c^2 a Tr_{\partial\Omega} u \;Tr_{\partial\Omega} \phi + \nu a Tr_{\partial\Omega}\partial_t u\; Tr_{\partial\Omega} \phi{\rm d} \mu \dt\label{Fu}\\
&\int_0^{+\infty}\int_{\Omega} -\alpha (u \partial_t^2 u) \phi -\alpha ( \partial_t u)^2 \phi+f\phi\;d\lambda\dt.\nonumber
\end{align} 
 In addition, 
 for $u\in L^2([0,+\infty[;L^2(\Omega^*))$, we define for all $\phi\in L^2([0,+\infty[;H^1(\Omega^*))$
\begin{equation}\label{Fmbar}
\overline{F}_m[u,\phi]=\left\lbrace\begin{array}{ll}
F_m[u,\phi] & \hbox{ if } u\in H(\Omega^*),\\
+\infty & \hbox{otherwise}
\end{array}\right.
\end{equation}
and
\begin{equation}\label{Fbar}
\overline{F}[u,\phi]=\left\lbrace\begin{array}{ll}
F[u,\phi] & \hbox{ if } u\in H(\Omega^*),\\
+\infty & \hbox{otherwise}.
\end{array}\right.
\end{equation}
\begin{remark}\label{defweaksolutionKuz}
We see that $u$ is a weak solution of the Westervelt problem~(\ref{CauchypbWesdirhom}) on $[0,+\infty[ \times \Omega $ in the sense of Theorem~\ref{ThWPWestGlob} if there hold
\begin{itemize}
\item $u\in X$ with the space $X$ defined in~(\ref{solspacedirhom});
\item for all $\phi \in L^2([0,+\infty[;V(\Omega))$
$F[u,\phi]=0,$
where $F$ is defined in~(\ref{Fu});
\item $u(0)=u_0$ and $u_t(0)=u_1$ on $\Omega$.
\end{itemize}
\end{remark}
The expression $F[u,\phi]=0$ can be obtained multiplying the Westervelt equation from system~(\ref{CauchypbWesdirhom}) by $\phi \in X$ integrating on $[0,+\infty[\times \Omega$ and doing integration by parts taking into account the boundary conditions.
In the same way with $F_m[u,\phi]$ given by Eq.~(\ref{Fnu}) we can define the weak solution of problem~(\ref{Kuzeqrobinpref}).

In order to state our main result, we also need to recall the notion of $M-convergence$ of functionals introduced in Ref.~\cite{Mosco}.
\begin{definition}
A sequence of functionals $G^m:H\rightarrow (-\infty,+\infty]$ is said to $M$-converge to a functional $G:H\rightarrow (-\infty,+\infty]$ in a Hilbert space $H$, if
\begin{enumerate}
\item(lim sup condition) For every $u\in H$ there exists $u_m$ converging strongly in $H$ such that
\begin{equation}
\overline{\lim} G^m[u_m]\leq G[u],\;\;\;\hbox{as}\;m\rightarrow +\infty.
\end{equation}
\item(lim inf condition) For every $v_m$ converging weakly to $u$ in $H$
\begin{equation}
\underline{\lim} G^m[v_m]\geq G[u],\;\;\;\hbox{as}\;m\rightarrow +\infty.
\end{equation}
\end{enumerate}
\end{definition}
Because of the quadratic nonlinearity of the Westervelt equation to be controlled for weak solutions, we consider domains in $\R^n$ only with $n=2$ or $3$. For the linear problem, it is possible to work with higher dimensions too.
The main result is the following theorem.
\begin{theorem}\label{Mconv}
Let $\Omega$ be a fractal domain of $\R^2$ or $\R^3$ and $(\Omega_m)_{m\in \N^*}$ be the prefractal polyhedral sequence described and defined previously and satisfying Assumptions~\ref{a-face} and \ref{a-uniuni}, all included in a Sobolev admissible domain $\Omega^*$.
For $\phi\in L^2([0,+\infty[;H^1(\Omega^*))$, and $a_m=a\sigma_m$, where $\sigma_m $ defined in~(\ref{defsigmn}), the sequence of functionals $u\mapsto \overline{F}_m[u,\phi]$ defined in~(\ref{Fmbar}), $M$-converges in $L^2([0,+\infty[;L^2(\Omega^*))$, to the following functional $u\mapsto \overline{F}[u,\phi]$ defined in~(\ref{Fbar}) as $m\rightarrow +\infty$. 

Moreover, for all $\phi\in L^2([0,+\infty[;H^1(\Omega^*))$ if $v_m\rightharpoonup u$ in $H(\Omega^*)$ defined in~(\ref{Moscospace}), then 
$$ F_m[v_m,\phi]\underset{m\rightarrow +\infty}{\longrightarrow}F[u,\phi],$$
where $F_m$ and $F$ are defined by equations~(\ref{Fnu}) and~(\ref{Fu}) respectively.
\end{theorem}
\begin{remark}
 If $(\Omega_m)_{m\in \N^*}$ is a monotone increasing sequence up to $\Omega$, $i.e.$ $\Omega_m\subset \Omega$ for all $m$, then it is not necessary to take $\Omega^*$ different to $\Omega$, it is sufficient to take $\Omega^*=\Omega$. In all cases, thanks to Theorem~\ref{extR2prefract}, 
 functions $v_m(t)\in H^1(\Omega_m)$ can be uniformly on $m$ extended to the functions $E_{\Omega_m} v_m(t)\in H^1(\R^n)$ and after it we work with their restrictions on $\Omega^*$: $\left[E_{\Omega_m} v_m(t)\right]|_{\Omega^*}\in H^1(\Omega^*)$. By the same way, for $u(t)\in H^1(\Omega)$ we consider if $\Omega\ne \Omega^*$ $\left[E_{\Omega} u(t)\right]|_{\Omega^*}\in H^1(\Omega^*)$. To avoid complicated notations we work directly with functions from $H^1(\Omega^*)$. 
 
 We also make the attention that we don't impose on $(\Omega_m)_{m\in \N^*}$ any restriction to be monotone, but only to satisfy Assumptions~\ref{a-face} and \ref{a-uniuni} and, by the fractal approximation, to converge to $\Omega$ in the sense of the characteristic functions: $\|\mathds{1}_{\Omega_m}-\mathds{1}_{\Omega}\|_{L^1(\Omega^*)}\to 0$ for $m\to +\infty$.
 \end{remark}

\begin{proof}
We consider $\phi\in L^2([0,+\infty[;H^1(\Omega^*))$.

\textbf{Proof of "lim sup" condition.} Without loss of generality, let us take directly a fixed $u\in H(\Omega^*)$ and define $v_m=u$ for all $m$. Hence $(v_m)_{m\in \N^*}$ is strongly converging sequence in $L^2([0,+\infty[;L^2(\Omega^*))$. Thus by the definition of functionals $\overline{F_m}[u,\phi]$ and $\overline{F}[u,\phi]$, they are equal respectively to 
$F_m[u,\phi]$ and $F[u,\phi]$, which are well defined (and hence are finite).
As by our construction $\Omega_m\rightarrow \Omega$ 
for $m\to +\infty$ in the sense of the characteristic functions
 and $u \in H(\Omega^*)$, for the linear terms in~(\ref{Fnu}) integrated over $\Omega_m$ to pass to the limit we can directly apply the 
 dominated convergence theorem for $m\rightarrow +\infty$
\begin{multline}
\int_0^{+\infty}\int_{\Omega_m} \partial_t^2 u \phi+c^2 \nabla u \nabla \phi+\nu \nabla \partial_t u \nabla \phi\dx\dt\\
\rightarrow \int_0^{+\infty}\int_{\Omega} \partial_t^2 u \phi+c^2 \nabla u \nabla \phi+\nu \nabla \partial_t u \nabla \phi\dx\dt.\label{limsupconv1}
\end{multline}
Indeed, knowing that $u\in H(\Omega^*)$ and $\phi\in L^2([0,+\infty[;H^1(\Omega^*))$ by H\"older's inequality we have
\begin{align*}
\int_0^{+\infty}\int_{\Omega^*} \vert \partial_t^2 u \phi\vert\dx\dt
\leq & \Vert \partial_t^2 u \Vert_{L^2([0,+\infty[;L^2(\Omega^*))} \Vert \phi\Vert_{L^2([0,+\infty[;L^2(\Omega^*))}<+\infty,\\
\int_0^{+\infty}\int_{\Omega^*} \vert c^2 \nabla u \nabla \phi\vert\dx\dt
\leq & c^2 \Vert \nabla u \Vert_{L^2([0,+\infty[;L^2(\Omega^*))} \Vert\nabla \phi\Vert_{L^2([0,+\infty[;L^2(\Omega^*))}<+\infty,\\
\int_0^{+\infty}\int_{\Omega^*} \vert \nu \varepsilon \nabla \partial_t u \nabla \phi\vert\dx\dt
\leq & \nu \Vert \nabla \partial_t u \Vert_{L^2([0,+\infty[;L^2(\Omega^*))} \Vert \nabla\phi\Vert_{L^2([0,+\infty[;L^2(\Omega^*))}<+\infty.
\end{align*}
To pass to the limit for the nonlinear terms integrated over $\Omega_m$ we also apply the dominated convergence theorem
\begin{align}
\int_0^{+\infty}\int_{\Omega_m} -\alpha (u \partial_t^2 u) \phi -\alpha ( \partial_t u)^2 \phi\dx\dt
\rightarrow \int_0^{+\infty}\int_{\Omega} -\alpha (u \partial_t^2 u) \phi -\alpha ( \partial_t u)^2 \phi\dx\dt.\label{limsupconv1bis}
\end{align}
More precisely, we successively apply H\"older's inequality and the Sobolev embeddings to control
\begin{align*}
\int_0^{+\infty}\int_{\Omega^*} \vert (u & \partial_t^2 u) \phi\vert \dx\dt\\
\leq & \Vert u\Vert_{L^{\infty}([0,+\infty[;L^4(\Omega^*))}\Vert \partial^2_t u\Vert_{L^2([0,+\infty[;L^2(\Omega^*))} \Vert \phi\Vert_{L^2([0,+\infty[;L^4(\Omega^*))}\\
\leq & C \Vert u\Vert_{H^1([0,+\infty[;H^1(\Omega^*))}\Vert \partial^2_t u\Vert_{L^2([0,+\infty[;L^2(\Omega^*))} \Vert \phi\Vert_{L^2([0,+\infty[;H^1(\Omega^*))}<+\infty
\end{align*}
and 
\begin{align*}
\int_0^{+\infty}\int_{\Omega^*} \vert & ( \partial_t u)^2 \phi\vert\dx\dt \\
\leq & \Vert \partial_t u\Vert_{L^{\infty}([0,+\infty[;L^2(\Omega^*))} \Vert \partial_t u\Vert_{L^{2}([0,+\infty[;L^4(\Omega^*))} \Vert \phi\Vert_{L^2([0,+\infty[;L^4(\Omega^*))}\\
\leq & C \Vert \partial_t u\Vert_{H^1([0,+\infty[;L^2(\Omega^*))} \Vert \partial_t u\Vert_{L^{2}([0,+\infty[;H^1(\Omega^*))} \Vert \phi\Vert_{L^2([0,+\infty[;H^1(\Omega^*))}<+\infty.
\end{align*}

Let $(\phi_j)_{j\in \N^*}\subset C^{\infty}([0,+\infty[\times\Omega^*)$ be a bounded sequence converging to $\phi$:
$$\phi_j\underset{j\rightarrow +\infty}{\rightarrow}\phi\quad\hbox{in}\quad L^2([0,+\infty[,H^1(\Omega^*)).$$
Thus we can express the difference between the boundary therms as follows
\begin{align}
\int_0^{+\infty} \int_{K_m}& \sigma_m Tr_{\partial\Omega_m} \partial_t u \;Tr_{\partial\Omega_m} \phi \ds\dt-\int_0^{+\infty} \int_{K} Tr_{\partial\Omega} \partial_t u \;Tr_{\partial\Omega} \phi d\mu \dt\nonumber\\
=&\int_0^{+\infty} \int_{K_m} \sigma_m Tr_{\partial\Omega_m} \partial_t u \;Tr_{\partial\Omega_m} (\phi-\phi_j) \ds\dt\label{limsupconvbound1}\\
&+\int_0^{+\infty} \int_{K_m} \sigma_m Tr_{\partial\Omega_m} \partial_t u \;Tr_{\partial\Omega_m} \phi_j \ds\dt-\int_0^{+\infty} \int_{K} Tr_{\partial\Omega} \partial_t u \;Tr_{\partial\Omega} \phi_j {\rm d} \mu \dt\nonumber\\
&-\int_0^{+\infty} \int_{K} Tr_{\partial\Omega} \partial_t u \;Tr_{\partial\Omega} (\phi-\phi_j) {\rm d} \mu \dt\nonumber.
\end{align}

Thanks to Theorems~\ref{contrl2trpref} and~\ref{ThGBesov}, we estimate the first integral in~(\ref{limsupconvbound1}) using the Cauchy-Schwartz inequality by
\begin{multline*}
\left| \int_0^{+\infty} \int_{K_m} \sigma_m Tr_{\partial\Omega_m} \partial_t u \;Tr_{\partial\Omega_m} (\phi-\phi_j) \ds\dt \right| \\
\leq C \Vert \partial_t u\Vert_{L^2([0,+\infty[,H^1(\Omega^*))} \Vert \phi-\phi_j\Vert_{L^2([0,+\infty[,H^1(\Omega^*))},
\end{multline*}
with a constant $C>0$ independent on $m$. Therefore, for all $\eps>0$ there exists $j_1\in \N^*$ such that for all $j\geq j_1$ and all $m\in \mathbb{N}$
\begin{equation}\label{limsupconvbound2}
\left\vert \int_0^{+\infty} \int_{K_m} \sigma_m Tr_{\partial\Omega_m} \partial_t u \;Tr_{\partial\Omega_m} (\phi-\phi_j) \ds\dt\right\vert\leq \frac{\varepsilon}{3}.
\end{equation}
In the same way by Theorems~\ref{THcontrl2tr} and~\ref{ThGBesov} we can show that there exists $j_2\in \N^*$ such that for all $j\geq j_2$ 
\begin{equation}\label{limsupconvbound3}
\left\vert \int_0^{+\infty} \int_{K} Tr_{\partial\Omega} \partial_t u \;Tr_{\partial\Omega} (\phi-\phi_j) {\rm d} \mu\dt\right\vert\leq \frac{\varepsilon}{3}.
\end{equation}
Let us now fix $j=\max (j_1,j_2)$. Given the regularity of $\phi_j$, we have $$\partial_t u\phi_j\in L^2([0,+\infty[,H^1(\Omega^*)).$$ So by Theorem~\ref{convtrace} for almost all time $t\in [0,+\infty[$ we find
$$\int_{K_m} \sigma_m Tr_{\partial\Omega_m} \partial_t u \;Tr_{\partial\Omega_m} \phi_j \ds- \int_{K} Tr_{\partial\Omega} \partial_t u \;Tr_{\partial\Omega} \phi_j {\rm d} \mu\underset{m\rightarrow +\infty}{\rightarrow} 0.$$
Moreover by~\eqref{contrl2trpref} and Theorem~\ref{ThGBesov}
\begin{align*}
\left\vert \int_{K_m} \sigma_m Tr_{\partial\Omega_m} \partial_t u \;Tr_{\partial\Omega_m} \phi_j \ds\dt\right\vert\leq 
C \Vert\partial_t u\Vert_{H^1(\Omega^*)} \Vert \phi_j \Vert_{H^1(\Omega^*)},
\end{align*}
with a constant $C>0$ independent on $m$. We notice that since $u\in H(\Omega^*)$ $$\Vert\partial_t u\Vert_{H^1(\Omega^*)} \Vert \phi_j \Vert_{H^1(\Omega^*)}\in L^1([0,\infty[).$$
Consequently, by the dominated convergence theorem
\begin{equation}\label{limsupconvbound4}
\left|\int_0^{+\infty} \int_{K_m} \sigma_m Tr_{\partial\Omega_m} \partial_t u \;Tr_{\partial\Omega_m} \phi_j \ds\dt-\int_0^{+\infty} \int_{K} Tr_{\partial\Omega} \partial_t u \;Tr_{\partial\Omega} \phi_j {\rm d} \mu \dt\right|\underset{m\rightarrow +\infty}{\rightarrow} 0.
\end{equation}
Thus, putting together~(\ref{limsupconvbound2}),~(\ref{limsupconvbound3}) and~(\ref{limsupconvbound4}) in~(\ref{limsupconvbound1}), we finally obtain that
\begin{equation}\label{limsupconvbound5}
\int_0^{+\infty} \int_{K_m} \sigma_m Tr_{\partial\Omega_m} \partial_t u \;Tr_{\partial\Omega_m} \phi \ds\dt-\int_0^{+\infty} \int_{K} Tr_{\partial\Omega} \partial_t u \;Tr_{\partial\Omega} \phi {\rm d} \mu \dt\underset{m\rightarrow +\infty}{\rightarrow} 0.
\end{equation}
In the same way we prove that
\begin{equation}\label{limsupconvbound6}
\int_0^{+\infty} \int_{K_m} \sigma_m Tr_{\partial\Omega_m} u \;Tr_{\partial\Omega_m} \phi \ds\dt-\int_0^{+\infty} \int_{K} Tr_{\partial\Omega} u \;Tr_{\partial\Omega} \phi {\rm d} \mu \dt\underset{m\rightarrow +\infty}{\rightarrow} 0.
\end{equation}
By using~(\ref{limsupconv1}),~(\ref{limsupconv1bis}),~(\ref{limsupconvbound5}),~(\ref{limsupconvbound6}) and the fact that by the dominated convergence theorem for $f\in L^2([0,+\infty[;L^2(\Omega))$
$$\int_{0}^{+\infty}\int_{\Omega_m}f\phi dx\dt\underset{m\rightarrow +\infty}{\rightarrow} \int_{0}^{+\infty}\int_{\Omega}f\phi \dx\dt,$$
we conclude that for all $\phi\in L^2([0,+\infty[,H^1(\Omega))$
$$F_m[u,\phi]\underset{m\rightarrow +\infty}{\longrightarrow} F[u,\phi].$$
This proves the "lim sup" condition since the infinite case obviously holds.

\textbf{Proof of the "lim inf" condition.}
Now, let $(v_m)_{m\in \N^*}$ be a bounded sequence in $H(\Omega^*)$ such that
$$v_m \rightharpoonup u\;\;\;\hbox{in}\;H(\Omega^*)\quad m\to +\infty.$$
Then by definition of $H(\Omega^*)$ in~(\ref{Moscospace}), it follows that
\begin{equation}\label{convdistr0}
\partial_t^2 v_m \rightharpoonup \partial_t^2 u\;\;\;\hbox{in}\; L^2([0,+\infty[;L^2(\Omega^*)),
\end{equation}
\begin{equation}\label{convdistr1}
\partial_t v_m \rightharpoonup \partial_t u,\; \nabla\partial_t v_m \rightharpoonup \nabla\partial_t u \;\;\;\hbox{in}\; L^2([0,+\infty[;L^2(\Omega^*)),
\end{equation}
and
\begin{equation}\label{convdistr2}
v_m\rightharpoonup u,\hbox{ } \nabla v_m \rightharpoonup \nabla u \;\;\;\hbox{in}\;L^2([0,+\infty[;L^2(\Omega^*)).
\end{equation}
Moreover, working in $\R^n$ with dimension $n\leq 3$, by Theorem~\ref{thmsobolembadm} it is possible to chose any $2 \leq p<6$ ensuring the compactness of the embedding $L^2([0,+\infty[;H^1(\Omega^*))\subset\subset L^2([0,+\infty[;L^p(\Omega^*))$. For higher dimension the desired assertion with $p\ge 2$ fails. So for $2 \leq p<6$
\begin{equation}\label{convdistr4}
 v_m\rightarrow u,\hbox{ }\partial_t v_m \rightarrow \partial_t u \hbox{ in } L^2([0,+\infty[;L^p(\Omega^*)).
\end{equation}
From the compact embedding of $H^1(\Omega^*)$ in $H^{\sigma}(\Omega^*)$ ($\frac{1}{2}<\sigma<1$) we also have that 
\begin{equation}\label{convdistr5}
 v_m\rightarrow u,\hbox{ }\partial_t v_m \rightarrow \partial_t u \hbox{ in } L^2([0,+\infty[;H^{\sigma}(\Omega^*)).
\end{equation}
Let $\phi\in L^2([0,\infty[,H^1(\Omega^*))$, we want to show that 
$$F_m[v_m,\phi] \underset{m\rightarrow+\infty}{\longrightarrow} F[u,\phi].$$
We start by studying the convergence of the terms with $\int_0^{+\infty}\int_{\Omega_m}$:
\begin{multline*}
\Big\vert \int_0^{+\infty}\int_{\Omega_m} \partial_t^2 v_m \phi \dx\ds- \int_0^{+\infty}\int_{\Omega} \partial_t u \partial_t \phi \dx\ds\Big\vert \leq \Big\vert \int_0^{+\infty}\int_{\Omega^*} \partial_t^2 v_m \mathbbm{1}_{\Omega_m} \phi\dx\ds\\
\;\;\;\;-\int_0^{+\infty} \int_{\Omega^*} \partial_t^2 v_m \mathbbm{1}_{\Omega}\phi\dx\ds\Big\vert
+\Big\vert \int_0^{+\infty}\int_{\Omega^*} \partial_t^2 v_m \mathbbm{1}_{\Omega} \phi\dx\ds
\;\;\;\;-\int_0^{+\infty} \int_{\Omega^*} \partial_t^2 u \mathbbm{1}_{\Omega} \phi\dx\ds\Big\vert.
\end{multline*}
The second term on the right hand side tends to zero as $m\rightarrow+\infty$ by~(\ref{convdistr0}) as $\mathbbm{1}_{\Omega}\partial_t\phi\in L^2([0,+\infty[;L^2(\Omega^*))$.
For the first term
$$\Big\vert \int_0^{+\infty}\int_{\Omega^*} \partial_t^2 v_m (\mathbbm{1}_{\Omega_m}-\mathbbm{1}_{\Omega}) \phi\dx\ds\Big\vert\leq \Vert (\mathbbm{1}_{\Omega_m}-\mathbbm{1}_{\Omega})\phi\Vert_{L^2([0,+\infty[\times\Omega^*) }\Vert \partial_t^2 v_m\Vert_{L^2([0,+\infty[\times\Omega^*)},$$
but $\Vert \partial_t^2 v_m\Vert_{L^2([0,+\infty[\times\Omega^*)}$ is bounded for all $m$ by~(\ref{convdistr0}) and by the dominated convergence theorem
$$\Vert (\mathbbm{1}_{\Omega_m}-\mathbbm{1}_{\Omega})\phi\Vert_{L^2([0,+\infty[\times\Omega^*) }\underset{m\rightarrow+\infty}{\longrightarrow} 0.$$
Then for $m\rightarrow+\infty$
$$\int_0^{+\infty}\int_{\Omega_m} \partial_t^2 v_m \phi \dx\ds\rightarrow \int_0^{+\infty}\int_{\Omega} \partial_t^2 u \phi \dx\ds.$$
 Using~(\ref{convdistr2}) we can deduce in the same way
\begin{multline}
\int_0^{+\infty}\int_{\Omega_m} \partial_t^2 v_m \phi+c^2 \nabla v_m \nabla \phi+\nu \varepsilon \nabla\partial_t v_m \nabla \phi\; \dx\dt\\
\underset{m\rightarrow+\infty}{\longrightarrow} \int_0^{+\infty}\int_{\Omega} \partial_t^2 u \phi+c^2 \nabla u \nabla \phi+\nu \varepsilon \nabla \partial_t u \nabla \phi \dx\dt.\label{liminfconv1}
\end{multline}
For the quadratic terms we have
\begin{multline}
\left\vert \int_0^{+\infty}\int_{\Omega_m} (v_m \partial_t^2 v_m) \phi dx dt-\int_0^{+\infty}\int_{\Omega} (u \partial_t^2 u) \phi dx dt\right\vert \\\leq \left\vert \int_0^{+\infty}\int_{\Omega_m} (v_m \partial_t^2 v_m) \phi dx dt - \int_0^{+\infty}\int_{\Omega} (v_m \partial_t^2 v_m) \phi dx dt\right\vert\\
+\left \vert\int_0^{+\infty}\int_{\Omega} ( v_m \partial_t^2 v_m) \phi dx dt -\int_0^{+\infty}\int_{\Omega} (u \partial_t^2 u) \phi dx dt\right\vert.\label{liminfconvquadstep1}
\end{multline}
To show that the first term on the right hand side tends to $0$ for $m\rightarrow +\infty$ we use the fact that by H\"older's inequality
\begin{align*}
\left\vert \int_0^{+\infty}\right. &\int_{\Omega_m} (v_m \partial_t^2 v_m) \phi dx dt -\left. \int_0^{+\infty}\int_{\Omega} (v_m \partial_t^2 v_m) \phi dx dt\right\vert\\
\leq & \Vert (\mathbbm{1}_{\Omega_m}-\mathbbm{1}_{\Omega})\phi\Vert_{L^2([0,+\infty[;L^4(\Omega^*)) } \Vert v_m\Vert_{L^{\infty}([0,+\infty[;L^4(\Omega^*))}\Vert \partial_t^2 v_m \Vert_{L^2([0,+\infty[;L^2(\Omega^*))}.
\end{align*}
Using the Sobolev embeddings we have for all $m$
\begin{align*}
\Vert v_m\Vert_{L^{\infty}([0,+\infty[;L^4(\Omega^*))} & \Vert \partial_t^2 v_m \Vert_{L^2([0,+\infty[;L^2(\Omega^*))}\\
\leq & C \Vert v_m\Vert_{H^1([0,+\infty[;H^1(\Omega^*))}\Vert \partial_t^2 v_m \Vert_{L^2([0,+\infty[;L^2(\Omega^*))} \leq K
\end{align*}
with a constant $K>0$ independent on $m$, as $(v_m)_{m\in \N^*}$ is weakly convergent in $H(\Omega^*)$. Moreover, as by the Sobolev embedding we have 
$$\phi\in L^2([0,+\infty[;H^1(\Omega^*))\subset\subset L^2([0,+\infty[;L^4(\Omega^*)),$$ 
then by the dominated convergence theorem we obtain
$$\Vert (\mathbbm{1}_{\Omega_m}-\mathbbm{1}_{\Omega})\phi\Vert_{L^2([0,+\infty[;L^4(\Omega^*)) }\underset{m\rightarrow+\infty}{\longrightarrow} 0.$$
So 
\begin{equation}\label{liminfconvquadstep2}
\left\vert \int_0^{+\infty}\int_{\Omega_m} (v_m \partial_t^2 v_m) \phi dx dt- \int_0^{+\infty}\int_{\Omega} (v_m \partial_t^2 v_m) \phi dx dt\right\vert \underset{m\rightarrow+\infty}{\longrightarrow} 0.
\end{equation}
Now we consider 
$$\left \vert\int_0^{+\infty}\int_{\Omega} ( v_m \partial_t^2 v_m) \phi dx dt -\int_0^{+\infty}\int_{\Omega} (u \partial_t^2 u) \phi dx dt\right\vert.$$
We see that
\begin{align*}
\Vert v_m \phi - u \phi\Vert_{L^2([0,+\infty[;L^2(\Omega))}=& \int_{0}^{+\infty}\Vert (v_m-u)\phi\Vert_{L^2(\Omega)}^2ds.\\
\intertext{Consequently, by the Young inequality}
\Vert v_m \phi - u \phi\Vert_{L^2([0,+\infty[;L^2(\Omega))} \leq & \int_{0}^{+\infty} \Vert v_m -u\Vert_{L^3(\Omega)}^2 \Vert \phi \Vert_{L^6(\Omega)}^2ds
\intertext{and by the Sobolev embeddings we find}
\Vert v_m \phi - u \phi\Vert_{L^2([0,+\infty[;L^2(\Omega))} \leq & K \int_{0}^{+\infty} \Vert v_m -u\Vert_{H^1(\Omega)}^2 \Vert \phi \Vert_{H^1(\Omega)}^2ds\\
\Vert v_m \phi - u \phi\Vert_{L^2([0,+\infty[;L^2(\Omega))} \leq & K \Vert v_m -u\Vert_{L^{\infty}([0,+\infty[; H^1(\Omega)}^2 \Vert \phi\Vert_{L^2([0,+\infty[;L^2(\Omega))}^2.
\end{align*}
Here $K>0$ is a general constant independing on $m$.
But we have
$v_m\rightharpoonup u$ in $H^1([0,+\infty[;H^1(\Omega^*))\subset\subset L^{\infty}([0,+\infty[;H^1(\Omega))$, so $v_m\rightarrow u$ in $L^{\infty}([0,+\infty[;H^1(\Omega))$.
Hence 
$$ v_m\phi \rightarrow u\phi \hbox{ in }L^2([0,+\infty[;L^2(\Omega)).$$
Combining this strong convergence result with the weak convergence~(\ref{convdistr0}) we obtain 
\begin{equation}\label{liminfconvquadstep3}
\left \vert\int_0^{+\infty}\int_{\Omega} ( v_m \partial_t^2 v_m) \phi dx dt -\int_0^{+\infty}\int_{\Omega} (u \partial_t^2 u) \phi dx dt\right\vert\underset{m\rightarrow+\infty}{\longrightarrow} 0.
\end{equation}
Then~(\ref{liminfconvquadstep1}),(\ref{liminfconvquadstep2}) and~(\ref{liminfconvquadstep3}) allow us to conclude that
\begin{equation}\label{liminfconvquad1}
\left\vert \int_0^{+\infty}\int_{\Omega_m} (v_m \partial_t^2 v_m) \phi dx dt-\int_0^{+\infty}\int_{\Omega} (u \partial_t^2 u) \phi dx dt\right\vert \underset{m\rightarrow+\infty}{\longrightarrow} 0.
\end{equation}
Now we consider
\begin{multline}
\left\vert \int_0^{+\infty}\int_{\Omega_m} ( \partial_t v_m)^2 \phi dx dt-\int_0^{+\infty}\int_{\Omega} ( \partial_t u)^2 \phi dx dt\right\vert \\
\leq \left \vert \int_0^{+\infty}\int_{\Omega_m} ( \partial_t v_m)^2 \phi dx dt- \int_0^{+\infty}\int_{\Omega} ( \partial_t v_m)^2 \phi dx dt\right\vert\\
+\left\vert \int_0^{+\infty}\int_{\Omega} ( \partial_t v_m)^2 \phi dx dt -\int_0^{+\infty}\int_{\Omega} ( \partial_t u)^2 \phi dx dt\right\vert\label{liminfconvquad2step1}
\end{multline}
The first term goes to $0$ when $m$ goes to infinity in the same way that for the proof of~(\ref{liminfconvquadstep2}), moreover we have:
\begin{align*}
\left\vert \int_0^{+\infty}\int_{\Omega} ( \partial_t v_m)^2 \phi dx dt -\int_0^{+\infty}\int_{\Omega} ( \partial_t u)^2 \phi dx dt\right\vert=& \left\vert \int_0^{+\infty}\int_{\Omega} (\partial_t v_m-\partial_t u)(\partial_t v_m+\partial_t u)\phi dx dt\right\vert.
\end{align*}
By the Young inequality
\begin{align*}
\left\vert \int_0^{+\infty}\int_{\Omega}( ( \partial_t v_m)^2-( \partial_t u)^2 ) \phi dx dt \right\vert\leq & \int_0^{+\infty} \Vert \partial_t v_m-\partial_t u\Vert_{L^3(\Omega)}\Vert \partial_t v_m+\partial_t u\Vert_{L^2(\Omega)}\Vert \phi\Vert_{L^6(\Omega)}dt
\end{align*}
and by the Sobolev embeddings and the Cauchy-Schwarz inequality we find
\begin{multline*}
\left\vert \int_0^{+\infty}\int_{\Omega}( ( \partial_t v_m)^2-( \partial_t u)^2 ) \phi dx dt \right\vert\leq \Vert \partial_t v_m+\partial_t u\Vert_{L^{\infty}([0,+\infty[;L^2(\Omega))}\\\cdot\Vert\partial_t v_m-\partial_t u\Vert_{L^{2}([0,+\infty[;L^3(\Omega))}\Vert \phi\Vert_{L^2([0,+\infty[;H^1(\Omega))}.
\end{multline*}
By~(\ref{convdistr4}) $\Vert\partial_t v_m-\partial_t u\Vert_{L^{2}([0,+\infty[;L^3(\Omega))}\underset{m\rightarrow+\infty}{\longrightarrow} 0$ and, as
$$\partial_t v_m \rightharpoonup \partial_t u \hbox{ in } H^1([0,+\infty[;L^2(\Omega))\subset\subset L^{\infty}([0,+\infty[;L^2(\Omega)),$$
the numerical sequence $(\Vert \partial_t v_m+\partial_t u\Vert_{L^{\infty}([0,+\infty[;L^2(\Omega))})_{m\in \N^*}$ is bounded. Consequently
$$\left\vert \int_0^{+\infty}\int_{\Omega}( ( \partial_t v_m)^2-( \partial_t u)^2 ) \phi dx dt \right\vert\underset{m\rightarrow+\infty}{\longrightarrow} 0.$$
Coming back to~(\ref{liminfconvquad2step1}) we obtain
\begin{equation}\label{liminfconvquad2}
\left\vert \int_0^{+\infty}\int_{\Omega_m} ( \partial_t v_m)^2 \phi dx dt-\int_0^{+\infty}\int_{\Omega} ( \partial_t u)^2 \phi dx dt\right\vert \underset{m\rightarrow+\infty}{\longrightarrow} 0.
\end{equation}

Let us consider the boundary term 
\begin{align}
\int_0^{+\infty} \int_{K_m}& \sigma_m Tr_{\partial\Omega_m} \partial_t v_m \;Tr_{\partial\Omega_m} \phi \ds\dt-\int_0^{+\infty} \int_{K} Tr_{\partial\Omega} \partial_t u \;Tr_{\partial\Omega} \phi {\rm d} \mu \dt\nonumber\\
= & \int_0^{+\infty} \int_{K_m} \sigma_m Tr_{\partial\Omega_m} \partial_t (v_m-u) \;Tr_{\partial\Omega_m} \phi \ds\dt\label{liminfconvbound1}\\
&+\left(\int_0^{+\infty} \int_{K_m} \sigma_m Tr_{\partial\Omega_m} \partial_t u \;Tr_{\partial\Omega_m} \phi \ds\dt-\int_0^{+\infty} \int_{K} Tr_{\partial\Omega} \partial_t u \;Tr_{\partial\Omega} \phi {\rm d} \mu \dt\right).\nonumber
\end{align}
By~(\ref{limsupconvbound5}) we already have the convergence to zero of the second term in~(\ref{liminfconvbound1}). Now thanks to Theorems~\ref{THcontrl2tr} and~\ref{ThGBesov} we find
\begin{align*}
\vert\int_0^{+\infty} \int_{K_m} & \sigma_m Tr_{\partial\Omega_m} \partial_t (v_m-u) \;Tr_{\partial\Omega_m} \phi \ds\dt\vert\\
\leq & C \Vert E_{\mathbb{R}^2}( \partial_t v_m-\partial_t u)\Vert_{L^2([0,+\infty[,H^{\sigma}(\mathbb{R}^2))} \Vert E_{\mathbb{R}^2}\phi \Vert_{L^2([0,+\infty[,H^{1}(\mathbb{R}^2))}\\
\leq & C \Vert \partial_t v_m-\partial_t u\Vert_{L^2([0,+\infty[,H^{\sigma}(\Omega^*))} \Vert \phi \Vert_{L^2([0,+\infty[,H^{1}(\Omega^*))}
\end{align*}
with a constant $C>0$ independent on $m$. Then by~(\ref{convdistr5})
\begin{equation}\label{liminfconvbound2}
\left\vert\int_0^{+\infty} \int_{K_m} \sigma_m Tr_{\partial\Omega_m} \partial_t (v_m-u) \;Tr_{\partial\Omega_m} \phi \ds\dt\right\vert\underset{m\rightarrow +\infty}{\rightarrow}0.
\end{equation}
By~(\ref{liminfconvbound1}),~(\ref{liminfconvbound2}) and~(\ref{limsupconvbound5}) we result in
\begin{equation}\label{liminfconvbound3}
\int_0^{+\infty} \int_{K_m} \sigma_m Tr_{\partial\Omega_m} \partial_t v_m \;Tr_{\partial\Omega_m} \phi \ds\dt-\int_0^{+\infty} \int_{K} Tr_{\partial\Omega} \partial_t u \;Tr_{\partial\Omega} \phi {\rm d} \mu \dt\underset{m\rightarrow +\infty}{\rightarrow}0.
\end{equation} 
In the same way
\begin{equation}\label{liminfconvbound4}
\int_0^{+\infty} \int_{K_m} \sigma_m Tr_{\partial\Omega_m} v_m \;Tr_{\partial\Omega_m} \phi \ds\dt-\int_0^{+\infty} \int_{K} Tr_{\partial\Omega} u \;Tr_{\partial\Omega} \phi {\rm d} \mu \dt\underset{m\rightarrow +\infty}{\rightarrow}0.
\end{equation} 
So by~(\ref{liminfconv1}),~(\ref{liminfconvquad1}),~(\ref{liminfconvquad2}),~(\ref{liminfconvbound3}) and~(\ref{liminfconvbound4}) we have for all $\phi\in L^2([0,+\infty[;H^1(\Omega))$
$$F_m[v_m,\phi]\rightarrow F[u,\phi],$$
as $m\rightarrow+\infty$ and this conclude the proof.
\end{proof}

We finish by proving the weak convergence of the solutions of the Westervelt problem on the prefractal domains to the weak solution on the fractal domain.
\begin{theorem}\label{thmconv}
Let domains $\Omega$ and $\Omega_m$ in $\R^2$ or $\R^3$ be defined as previously satisfying Assumptions~\ref{a-face} and \ref{a-uniuni}, $\Omega^*$ be a Sobolev admissible domain such that $\Omega\subset\Omega^*$, $\forall m$ $\Omega_m\subset\Omega^*$ and
$$\partial\Gamma_{D,\Omega_m}=\partial\Gamma_{D,\Omega}=\partial\Gamma_{D,\Omega^*}.$$ 
For $g\in L^2(\Omega^*)$, let $u_0\in V(\Omega)$, $u_1\in V(\Omega)$, $\Delta u_0=g\vert_{\Omega}\in L^2(\Omega) $ in the sense of the Poisson problem~(\ref{PoissonDir1}) with $a>0$. In addition, let for all $m\in \N^*$ 
$u_{0,m}\in V(\Omega_m)$ and $u_{1,m}\in V(\Omega_m)$ with $ \Delta u_{0,m}=g\vert_{\Omega_m}\in L^2(\Omega_m)$ such that 
\begin{align*}
(E_{\mathbb{R}^2}u_{0,m})\vert_{\Omega}\underset{m\rightarrow+\infty}{\rightharpoonup} u_0 \hbox{ in }H^1(\Omega), \\
(E_{\mathbb{R}^2} u_{1,m})\vert_{\Omega} \underset{m\rightarrow+\infty}{\rightharpoonup} u_1\hbox{ in }H^1(\Omega).
\end{align*}
Then for $u_m\in X(\Omega_m)$, the weak solution of problem~(\ref{Kuzeqrobinpref}) on $\Omega_m$ associated to the initial conditions $u_{0,m}$ and $u_{1,m}$ in the sense of Theorem~\ref{ThWPWestGlob} with $a_m=a\sigma_m$, and $u\in X(\Omega)$, the weak solution of problem~(\ref{CauchypbWesdirhom}) on $\Omega$ in the sense of Theorem~\ref{ThWPWestGlob}, 
it follows that they are weak solutions in the sense of Remark~\ref{defweaksolutionKuz}
and 
$$(E_{\mathbb{R}^2}u_m)\vert_{\Omega^*}\rightharpoonup u^* \hbox{ in } H(\Omega^*)\hbox{ with } u^*\vert_{\Omega}=u,$$
where $H(\Omega^*)$ is defined in~(\ref{Moscospace}).
\end{theorem}
\begin{proof}
By the definitions of $u_m$ and $u$ respectively from Theorem~\ref{ThWPWestGlob} we have as a direct consequence that $u_m\in X(\Omega_m)$ and $u\in X(\Omega)$ are weak solutions in the sense of Remark~\ref{defweaksolutionKuz}. Therefore for all $\phi_1\in L^2([0,+\infty[;V(\Omega_m))$ and $\phi_2\in L^2([0,+\infty[;V(\Omega))$
$$F_m[u_m,\phi_1]=0\quad \hbox{and} \quad F[u,\phi_2]=0.$$
Extending with the help of Theorem~\ref{extR2prefract} we obtain
$$\Vert (E_{\mathbb{R}^2}u_m)\vert_{\Omega^*}\Vert_{H(\Omega^*)}\leq C \Vert u_m\Vert_{H(\Omega_m)}$$
with a constant $C>0$ independent on $m$.

By assumption, for $(\Omega_m)_{m\in\mathbb{N}}$ and $\Omega$ we have the same $\partial\Gamma_D$ fixed, and they are all $(\varepsilon,\delta)$-domains with fixed $\varepsilon$ and $\delta$. So we can apply Theorem~\ref{thmPoinc2}. 
After what we apply Theorems~\ref{thmlinfestL2},~\ref{dampedwaveeqregint1},~\ref{thmeqivcaudampwav},~\ref{ThWPWestGlob}. As in these theorems, the dependence of the constants on the domain only depends on the constant from the Poincar\'e's inequality, we obtain the existence of $r^*$ independent on $m$ in Theorem~\ref{ThWPWestGlob} such that if $r<r^*$ and
$$
\Vert f\Vert_{L^2([0,+\infty[;L^2(\Omega_m))} +\Vert \Delta u_{0,m}\Vert_{L^2(\Omega_m)}+\Vert u_{1,m}\Vert_{V(\Omega_m)}\leq \frac{\nu }{C_1}r,
$$
with $C_1>0$ independent on $m$, then
$$C \Vert u_m\Vert_{H(\Omega_m)} \leq \Vert u_m\Vert_{X(\Omega_m)}\leq 2r,$$
with $C>0$ independent on $m$.
Therefore,
$$\Vert (E_{\mathbb{R}^2}u_m)\vert_{\Omega^*}\Vert_{H(\Omega)}\leq K $$
with a constant $K>0$ independent on $m$, and consequently, there exits $u^*$ in $ H(\Omega^*)$ and a subsequence still denoted by $(E_{\mathbb{R}^2}u_m)\vert_{\Omega^*}$ such that
$$(E_{\mathbb{R}^2}u_m)\vert_{\Omega^*}\rightharpoonup u^*\quad \hbox{in}\quad H(\Omega^*).$$
Now for $m\in\mathbb{N}$ we define 
$$U_m:=\bigcap_{i=m}^{+\infty}\Omega_i\cap\Omega.$$
It is an increasing sequence of open sets with $U_m\uparrow\Omega$ for $m\to +\infty$.

We also define
$$V(U_m):=\lbrace u\in H^1(U_m)\vert\; \text{Tr} u=0\hbox{ on }\Gamma_{Dir,U_m}=(\cup_{i=m}^{\infty}\Gamma_{D,\Omega_i}\cup\Gamma_{D,\Omega})\cap \partial U_m\rbrace,$$
the closed set $W_m\subset\Omega^*$ such that $\partial W_m=\Gamma_{Dir,U_m}\cup\Gamma_{D,\Omega^*}$
and 
$$V_m(\Omega^*):=\lbrace \phi \in L^2([0,+\infty[;H^1(\Omega^*)\vert\;\;\; \phi\vert_{U_m}\in L^2([0,+\infty[;V(U_m))\hbox{ and }u=0\hbox{ on }W_m\rbrace.$$
Set 
$\phi\in V_M(\Omega^*)$, then for all $m\geq M$ $\phi\vert_{\Omega_m}\in L^2([0,+\infty[;V(\Omega_m)) $. Thus by Theorem~\ref{Mconv} we have 
$$0=F_m[(E_{\mathbb{R}^2}u_m)\vert_{\Omega^*},\phi]\rightarrow F[u^*,\phi].$$
Consequently for all $M\in \mathbb{N}$ and for all $\phi\in V_M(\Omega^*)$ 
$$F[u^*,\phi]=0.$$
But by definition of $U^m$ for $\phi\in L^2([0,+\infty[;V(\Omega))$ we can construct a sequence 
$$\phi_m\in V_m(\Omega^*)\subset L^2([0,+\infty[;V(\Omega)) $$ such that
$$ \phi_m\vert_{\Omega} \underset{m\rightarrow+\infty}{\rightarrow}\phi\hbox{ in }L^2([0,+\infty[,V(\Omega)).$$
Then for all $\phi\in L^2([0,+\infty[;V(\Omega))$
$$F[u^*,\phi]=0.$$
 By definition of $u_m$ we also have $u^*(0)=u_0$, $\Delta u^*(0)=\Delta u_0$ in $L^2(\Omega)$ and $\partial_t u^*(0)=u_1$ in $V(\Omega)$. 
Moreover 
$$u^*\in H(\Omega^*). $$
Thus we deduce $u^*\vert_{\Omega}=u$ which allows to conclude.
\end{proof}
\begin{remark}
Given the variational formulations~(\ref{Fnu}) and~(\ref{Fu}), it is also possible to consider the prefractal approximations not only for $\Gamma_{R,\Omega}$ but also for $\Gamma_{N,\Omega}$ and $\Gamma_{D,\Omega}$ simultaneously, which different fractals can describe. In this case, Theorem~\ref{Mconv} stays true, and we have an equivalent of Theorem~\ref{thmconv} with the help of Theorem~\ref{thmPoinc2} which ensures that the constants in the Poincar\'e's inequality can be taken independently on $m$. 
As particular examples in $\R^2$, Theorems~\ref{Mconv} and~\ref{thmconv} hold for the studied in Ref.~\cite{Capitanelli2} case of von Koch mixtures (see Appendix~\ref{subsecexKoch}) and for the Minkowski fractal.
\end{remark}
\section*{Acknowledgements}
 The authors are deeply grateful to Luke G. Rogers and the unanimous referee for many helpful comments. The additional thanks and warm thoughts are in memory of M. F. Sukhinin recently died. 

\appendix

 \section{Proof of Theorem~\ref{thmlinfestL2}}\label{AppDaners}
 As in Ref.~\cite{Daners} let us define for every $m\in \N^*$, $t\geq 1$ 
the function
\begin{equation}
G_{t,m}(\xi):=\left 
\lbrace\begin{array}{ll}
0 & \hbox{if }\xi\leq 0,\\

\xi^t & \hbox{if }\xi\in ]0,m[,\\
m^{t-1}u & \hbox{if }\xi\geq m,
\end{array}\right.
\end{equation}
which by its definition is piece wise smooth and has a bounded derivative. This implies that $G_{t,m}(u)\in V(\Omega)$ for $u\in V(\Omega)$ by Theorem 7.8 of Ref.~\cite{Gilbarg}. 
For some fixed $m\geq 1$ and $q\geq 2$ we introduce the following notations:
$$v:=G_{q-1,m}(u), \quad w:=G_{\frac{q}{2},m}(u).$$
Using again Theorem 7.8 in Ref.~\cite{Gilbarg} we obtain that
$$ \partial_{x_i}w\partial_{x_j}w=\left\{ \begin{array}{ll}
 \frac{q^2}{4(q-1)}\partial_{x_i}u\partial_{x_j}v, &\hbox{ if } u(x)\leq m\\
 \partial_{x_i}u\partial_{x_j}v, & \hbox{ if }u(x)\geq m.
 \end{array} \right.$$

Consequently we find 
\begin{align*}
\Vert \nabla w\Vert_{L^2(\Omega)}^2\leq & \frac{q^2}{4(q-1)} (\nabla u,\nabla v)_{L^2(\Omega)}\\
\leq & q [(\nabla u,\nabla v)_{L^2(\Omega)}+a\int_{\Gamma_R}Tr_{\Gamma_R}u\;Tr_{\Gamma_R}v dm_d]\\
\leq & q (f,v)_{L^2(\Omega)}\\
\leq & q \Vert f\Vert_{L^2(\Omega)} \Vert v\Vert_{L^2(\Omega)}.
\end{align*}
Using estimate~(\ref{L6est}) we obtain
$$\Vert w\Vert_{L^6(\Omega)}^2\leq C \Vert \nabla w\Vert_{L^2(\Omega)}^2 \leq C q \Vert f\Vert_{L^2(\Omega)} \Vert v\Vert_{L^2(\Omega)},$$
where $C>0$ depends only on $\Omega$ in the same way as in Proposition~\ref{propapl6}.
Then we use the fact that
$0\leq v\leq w^{\frac{2(q-1)}{q}} $ to deduce
\begin{equation}\label{Lqestimate}
\Vert w^{\frac{2}{q}}\Vert_{L^ {3q}(\Omega)}^q\leq C q \Vert f\Vert_{L^2(\Omega)} \Vert w^{\frac{2}{q}}\Vert_{L^{2(q-1)}(\Omega)}^{q-1}.
\end{equation}
Let us denote by $u^+$ and $u^-$ the positive and negative parts of $u$, $u^{\pm}:=\max(0,\pm u)$.
 The sequence of functions $w^{\frac{2}{q}}= [G_{\frac{q}{2},m}(u)]^{\frac{2}{q}}$ is increasing as $m$ increases and converges to $u^+$ as $m$ goes to infinity. Thus, if we take $\overline{u}=\frac{u^+}{M}$ with $M=C \Vert f\Vert_{L^2(\Omega)}$, from~(\ref{Lqestimate}) with the help of the monotone convergence theorem we have
\begin{equation}\label{Lqestimatebis}
\Vert \overline{u}\Vert_{L^ {3q}(\Omega)}^q\leq q \Vert \overline{u}\Vert_{L^{2(q-1)}(\Omega)}^{q-1}.
\end{equation}
We take $q_0=2$ and $q_{n+1}=1+ \eta q_n$ with $\eta=\frac{3}{2}$ for all $n\in\mathbb{N}$, what allows us thanks to estimate~(\ref{Lqestimatebis}) to find
$$\Vert \overline{u}\Vert_{L^ {3q_{n+1}}(\Omega)}^{q_{n+1}}\leq q_{n+1} \Vert \overline{u}\Vert_{L^{3q_n}(\Omega)}^{\eta q_n}.$$
From the last estimate we obtain by induction that
$$ \Vert \overline{u}\Vert_{L^ {3q_{n+1}}(\Omega)}\leq \left(\prod_{k=1}^{n+1} q_k^{\frac{\eta^{n+1-k}}{q_{n+1}}} \right)\Vert \overline{u}\Vert_{L^6(\Omega)}^{2\frac{\eta^{n+1}}{q_{n+1}}}.$$
As $\eta=\frac{3}{2} >1$ we see that $\eta\leq \frac{q_{n+1}}{q_n} \leq 2\eta $, which by induction implies that $q_{n+1}=4\eta^{n+1}-2$.
Consequently,
$$ \Vert \overline{u}\Vert_{L^ {3q_{n+1}}(\Omega)}\leq 2^{\sum_{k=1}^{n+1}\eta^{-k}} (2\eta)^{\frac{1}{2}\sum_{k=1}^{n+1}k\eta^{-k}}\Vert \overline{u}\Vert_{L^6(\Omega)}^{2\frac{\eta^{n+1}}{4\eta^{n+1}-2}}.$$
Since $\eta>1$ we can pass to the limit for $n\to +\infty$: 
$$\Vert \overline{u}\Vert_{L^{\infty}(\Omega)}\leq K \Vert \overline{u}\Vert_{L^6}^{\frac{1}{2}},$$
where $$K= 2^{\sum_{k=1}^{+\infty}\eta^{-k}} (2\eta)^{\frac{1}{2}\sum_{k=1}^{+\infty}k\eta^{-k}}<+\infty.$$
Taking into account that $$\Vert \overline{u}\Vert_{L^{\infty}(\Omega)}\leq K \vert \Omega\vert^{\frac{1}{12}}\Vert \overline{u}\Vert_{L^{\infty}(\Omega)}^{\frac{1}{2}},$$
we conclude in
$$\Vert \overline{u}\Vert_{L^{\infty}(\Omega)}\leq K^2 \vert \Omega\vert^{\frac{1}{6}}.$$
Finally, by definition of $\overline{u}$ we obtain 
$$\Vert u^+\Vert_{L^{\infty}(\Omega)}\leq C \Vert f\Vert_{L^2(\Omega)},$$
where $C>0$ depends only on $\Omega$ in the same way as in Proposition~\ref{propapl6}.
As $u^-=(-u)^+$, and by linearity $-u$ is the solution of the Poisson problem~(\ref{PoissonDir1}) with $f$ replaced by $-f$, then we also have 
$$\Vert u^-\Vert_{L^{\infty}(\Omega)}\leq C \Vert f\Vert_{L^2(\Omega)},$$
which finishes the proof.
 
\section{Scale irregular Koch curves and the Strong Open Set Condition}\label{subsecexKoch}
 Koch mixtures~\cite{Capitanelli2} can give a typical example of a self-similar fractal boundary in $\R^2$. %

We recall briefly some notations introduced in Section $2$ page 1223 of Ref.~\cite{Capitanelli2} for scale irregular Koch curves built on two families of contractive similitudes. Let $\mathcal{B}=\lbrace 1,2\rbrace$: for $a\in \mathcal{B}$ let $2<l_a<4$, and for each $a\in \mathcal{B}$ let
$$\Psi^{(a)}=\lbrace \psi_1^{(a)},\ldots, \psi_4^{(a)}\rbrace$$
be the family of contractive similitudes $\psi_i^{(a)}:\mathbb{C}\rightarrow \mathbb{C}$, $i=1, \ldots, 4$, with contraction factor $l_a^{-1}$ defined in Ref.~\cite{Capitanelli}.

Let $\Xi=\mathcal{B}^{\mathbb{N}}$; we call $\xi\in \Xi$ an environnent. We define the usual left shift $S$ on $\Xi$. For $\mathcal{O}\subset \mathbb{R}^2$, we set 
$$\Phi^{(a)}(\mathcal{O})=\bigcup_{i=1}^4 \psi_i^{(a)}(\mathcal{O})$$
and
$$\Phi^{(\xi)}_m(\mathcal{O})=\Phi^{(\xi_1)}\circ\cdots\circ\Phi^{(\xi_m)}(\mathcal{O}).$$
Let $K$ be the line segment of unit length with $A=(0,0)$ and $B=(1,0)$ as end points. We set, for each $m$ in $\mathbb{N}$, 
$$K^{(\xi),m}=\Phi^{(\xi)}_m(K).$$
 $K^{(\xi),m}$ is the so-called $m$-th prefractal curve.
The fractal $K^{(\xi)}$ associated with the environment sequence $\xi$ is defined by
$$K^{(\xi)}=\overline{\bigcup_{m=1}^{+\infty}\Phi^{(\xi)}_m(\Gamma)},$$
where $\Gamma=\lbrace A,B\rbrace$. For $\xi\in \Xi$, we set $i\vert m=(i_1,\ldots,i_m)$ and $\psi_{i\vert m}=\psi_{i_1}^{(\xi_1)}\circ \cdots\circ \psi_{i_m}^{(\xi_m)}$. We define the volume measure $\mu^{(\xi)}$ as the unique Radon measure on $K^{(\xi)}$ such that
$$\mu^{(\xi)}(\psi_{i\vert m}(K^{(S^m\xi)}))=\frac{1}{4^m}$$
(see Section 2 in Ref.~\cite{Barlow}) as, for each $a\in \mathcal{B}$, the family $\Phi^{(a)}$ has $4$ contractive similitudes.

The fractal set $K^{(\xi)}$ and the volume measure $\mu^{(\xi)}$ depend on the oscillations in the environment sequence $\xi$. We denote by $h_a^{(\xi)}(m)$ the frequency of the occurrence of $a$ in the finite sequence $\xi \vert m$, $m\geq 1$:
$$h_a^{(\xi)}(m)=\frac{1}{m}\sum_{i=1}^m 1_{\lbrace\xi_i=a\rbrace}, \hbox{ }a=1,2.$$
Let $p_a$ be a probability distribution on $\mathcal{B}$, and suppose that $\xi$ satisfies
$$h_a^{(\xi)}(m)\underset{m\rightarrow +\infty}{\longrightarrow} p_a ,$$
(where $0\leq p_a\leq 1,$ $p_1+p_2=1$) and
$$\vert h_a^{(\xi)}(m)-p_a\vert\leq\frac{C_0}{m}, \hbox{ }a=1,2,\hbox{ }(n\geq 1),$$
with some constant $C_0\geq 1$, that is, we consider the case of the fastest convergence of the occurrence factors.

Under these conditions, the measure $\mu^{(\xi)}$ has the property that there exist two positive constants $C_1$, $C_2$, such that (see Refs.~\cite{Mosco2,Mosco3}),
$$C_1 r^{d^{(\xi)}}\leq \mu^{(\xi)}(K^{(\xi)}\cap B_r(x))\leq C_2 r^{d^{(\xi)}} \;\;\;\hbox{for all } x\in K^{(\xi)},0<r\leq1,$$
where $B_r(x)\subset \mathbb{R}^2$ denotes the Euclidean ball of radius $r$ and centered at $x$ with
$$d^{(\xi)}=\frac{\ln 4}{p_1 \ln p_1+p_2 \ln p_2} .$$
According to Definition \ref{defdset}, it means that $K^{(\xi)}$ is a $ d^{(\xi)}$-set and the measure $\mu^{(\xi)}
$ is a $d^{(\xi)}-$ dimensional measure equivalent to the $d^{(\xi)}$-dimensional Hausdorff measure $m_{d^{(\xi)}}$.

 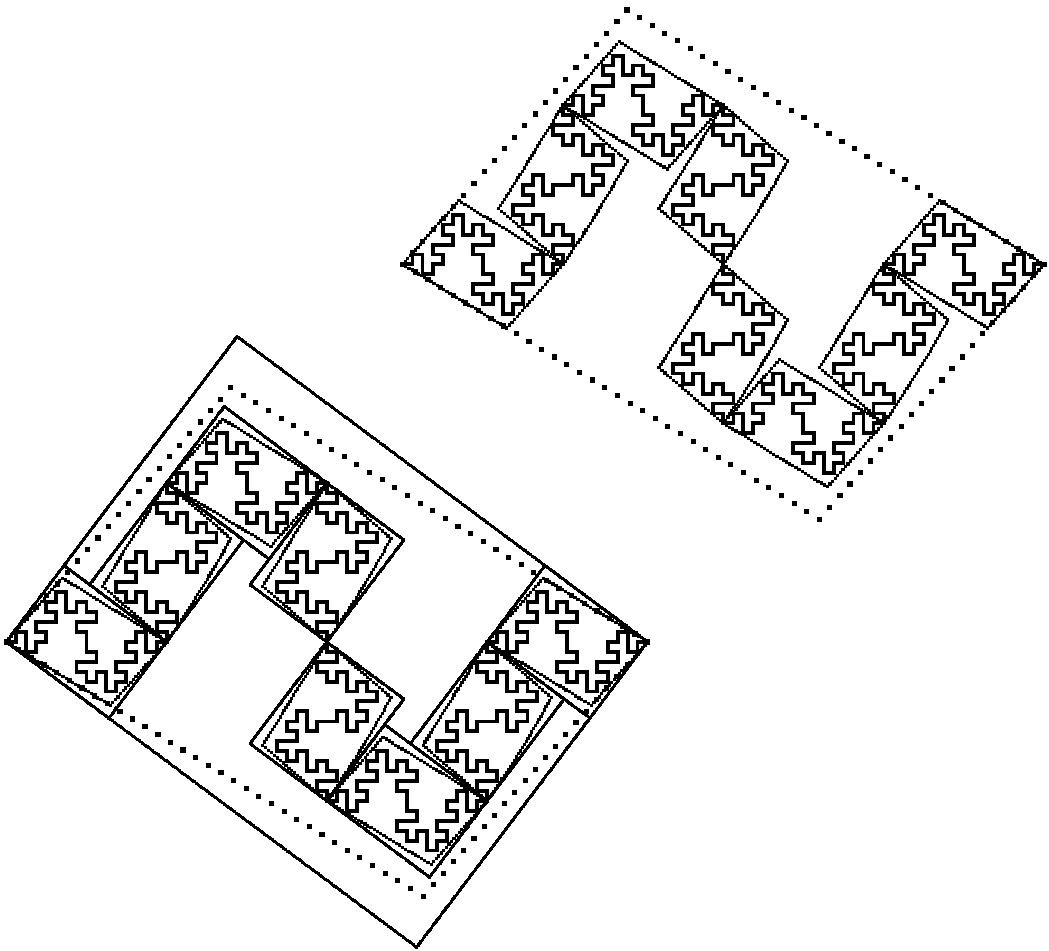
\begin{figure}[htb!]
 \ \hfill \begin{picture}(250,190)(0,-95)
 %
 \def\fh{\linethickness{1.0pt}\setlength{\unitlength}{0.375pt}\qbezier(0,0)(5,0)(10,0)
 \qbezier(10,0)(10,5)(10,10)
 \qbezier(10,10)(15,10)(20,10)
 \qbezier(20,10)(20,0)(20,-10)
 \qbezier(20,-10)(25,-10)(30,-10)
 \qbezier(30,-10)(30,-5)(30,0)
 \qbezier(30,0)(35,0)(40,0)
 \qbezier(0,0)(5,0)(10,0)}
 \def\fv{\linethickness{1.0pt}\setlength{\unitlength}{0.375pt}\qbezier(0,0)(0,5)(0,10)
 \qbezier(0,10)(-5,10)(-10,10)
 \qbezier(-10,10)(-10,15)(-10,20)
 \qbezier(-10,20)(00,20)(10,20)
 \qbezier(10,20)(10,25)(10,30)
 \qbezier(10,30)(5,30)(0,30)
 \qbezier(0,30)(0,35)(0,40)}
 \def\ffh{\linethickness{1.0pt}\setlength{\unitlength}{0.375pt}
 \put(00,00){\fh}
 \put(40,00){\fv}
 \put(40,40){\fh}
 \put(80,00){\fv}
 \put(80,-40){\fv}
 \put(120,-40){\fv}
 \put(80,-40){\fh}
 \put(120,00){\fh}}
 \def\ffv{\linethickness{1.0pt}\setlength{\unitlength}{0.375pt}
 \put(00,00){\fv}
 \put(-40,40){\fh}
 \put(-40,40){\fv}
 \put(-40,80){\fh}
 \put(00,80){\fh}
 \put(00,120){\fh}
 \put(40,80){\fv}
 \put(00,120){\fv}}
 \setlength{\unitlength}{1.5pt}
 \linethickness{1.5pt}
 \put(00,00){
 {\ffh}
 \linethickness{.75pt}
 \qbezier[22](0,0)(7,8)(14,16)
 \qbezier[33](40,0)(27,8)(14,16)
 \qbezier[22](40,0)(33,-8)(26,-16)
 \qbezier[33](0,0)(13,-8)(26,-16)
 }
 \put(40,00){
 {\ffv}
 \linethickness{.75pt}
 \qbezier[22](0,0)(-8,7)(-16,14)
 \qbezier[33](0,40)(-8,27)(-16,14)
 \qbezier[22](0,40)(8,33)(16,26)
 \qbezier[33](0,0)(9,13)(16,26)
 }
 \put(40,40){
 {\ffh}
 \linethickness{.75pt}
 \qbezier[22](0,0)(7,8)(14,16)
 \qbezier[33](40,0)(27,8)(14,16)
 \qbezier[22](40,0)(33,-8)(26,-16)
 \qbezier[33](0,0)(13,-8)(26,-16)
 }
 \put(80,00){
 {\ffv}
 \linethickness{.75pt}
 \qbezier[22](0,0)(-8,7)(-16,14)
 \qbezier[33](0,40)(-8,27)(-16,14)
 \qbezier[22](0,40)(8,33)(16,26)
 \qbezier[33](0,0)(9,13)(16,26)
 }
 \put(80,-40){
 {\ffv}
 \linethickness{.75pt}
 \qbezier[22](0,0)(-8,7)(-16,14)
 \qbezier[33](0,40)(-8,27)(-16,14)
 \qbezier[22](0,40)(8,33)(16,26)
 \qbezier[33](0,0)(9,13)(16,26)
 }
 \put(120,-40){
 {\ffv}
 \linethickness{.75pt}
 \qbezier[22](0,0)(-8,7)(-16,14)
 \qbezier[33](0,40)(-8,27)(-16,14)
 \qbezier[22](0,40)(8,33)(16,26)
 \qbezier[33](0,0)(9,13)(16,26)
 }
 \put(80,-40){
 {\ffh}
 \linethickness{.75pt}
 \qbezier[22](0,0)(7,8)(14,16)
 \qbezier[33](40,0)(27,8)(14,16)
 \qbezier[22](40,0)(33,-8)(26,-16)
 \qbezier[33](0,0)(13,-8)(26,-16)
 }
 \put(120,00){
 {\ffh}
 \linethickness{.75pt}
 \qbezier[22](0,0)(7,8)(14,16)
 \qbezier[33](40,0)(27,8)(14,16)
 \qbezier[22](40,0)(33,-8)(26,-16)
 \qbezier[33](0,0)(13,-8)(26,-16)
 }
 \setlength{\unitlength}{6pt}
 \qbezier[22](0,0)(7,8)(14,16)
 \qbezier[33](40,0)(27,8)(14,16)
 \qbezier[22](40,0)(33,-8)(26,-16)
 \qbezier[33](0,0)(13,-8)(26,-16)
 \end{picture}
 
 \vskip-15ex
 
 \begin{picture}(250,222)(0,-111) %
 \def\fh{\linethickness{1.0pt}\setlength{\unitlength}{0.375pt}\qbezier(0,0)(5,0)(10,0)
 \qbezier(10,0)(10,5)(10,10)
 \qbezier(10,10)(15,10)(20,10)
 \qbezier(20,10)(20,0)(20,-10)
 \qbezier(20,-10)(25,-10)(30,-10)
 \qbezier(30,-10)(30,-5)(30,0)
 \qbezier(30,0)(35,0)(40,0)
 \qbezier(0,0)(5,0)(10,0)}
 \def\fv{\linethickness{1.0pt}\setlength{\unitlength}{0.375pt}\qbezier(0,0)(0,5)(0,10)
 \qbezier(0,10)(-5,10)(-10,10)
 \qbezier(-10,10)(-10,15)(-10,20)
 \qbezier(-10,20)(00,20)(10,20)
 \qbezier(10,20)(10,25)(10,30)
 \qbezier(10,30)(5,30)(0,30)
 \qbezier(0,30)(0,35)(0,40)}
 \def\ffh{\linethickness{1.0pt}\setlength{\unitlength}{0.375pt}
 \put(00,00){\fh}
 \put(40,00){\fv}
 \put(40,40){\fh}
 \put(80,00){\fv}
 \put(80,-40){\fv}
 \put(120,-40){\fv}
 \put(80,-40){\fh}
 \put(120,00){\fh}}
 \def\ffv{\linethickness{1.0pt}\setlength{\unitlength}{0.375pt}
 \put(00,00){\fv}
 \put(-40,40){\fh}
 \put(-40,40){\fv}
 \put(-40,80){\fh}
 \put(00,80){\fh}
 \put(00,120){\fh}
 \put(40,80){\fv}
 \put(00,120){\fv}}
 \setlength{\unitlength}{1.5pt}
 \linethickness{1.5pt}
 \put(00,00){
 {\ffh}
 \linethickness{.75pt}
 \qbezier[22](0,0)(7,8)(14,16)
 \qbezier[33](40,0)(27,8)(14,16)
 \qbezier[22](40,0)(33,-8)(26,-16)
 \qbezier[33](0,0)(13,-8)(26,-16)
 \linethickness{.5pt}\setlength{\unitlength}{12pt}
 \qbezier(0,0)(0.9,1.2)(1.8,2.4)
 \qbezier(5,0)(3.4,1.2)(1.8,2.4)
 \qbezier(0,0)(1.6,-1.2)(3.2,-2.4)
 \qbezier(5,0)(4.1,-1.2)(3.2,-2.4)
 }
 \put(40,00){
 {\ffv}
 \linethickness{.75pt}
 \qbezier[22](0,0)(-8,7)(-16,14)
 \qbezier[33](0,40)(-8,27)(-16,14)
 \qbezier[22](0,40)(8,33)(16,26)
 \qbezier[33](0,0)(9,13)(16,26)
 \linethickness{.5pt}\setlength{\unitlength}{12pt}
 \qbezier(0,0)(-1.2,0.9)(-2.4,1.8)
 \qbezier(0,5)(-1.2,3.4)(-2.4,1.8)
 \qbezier(0,0)(1.2,1.6)(2.4,3.2)
 \qbezier(0,5)(1.2,4.1)(2.4,3.2)
 }
 \put(40,40){
 {\ffh}
 \linethickness{.75pt}
 \qbezier[22](0,0)(7,8)(14,16)
 \qbezier[33](40,0)(27,8)(14,16)
 \qbezier[22](40,0)(33,-8)(26,-16)
 \qbezier[33](0,0)(13,-8)(26,-16)
 \linethickness{.5pt}\setlength{\unitlength}{12pt}
 \qbezier(0,0)(0.9,1.2)(1.8,2.4)
 \qbezier(5,0)(3.4,1.2)(1.8,2.4)
 \qbezier(0,0)(1.6,-1.2)(3.2,-2.4)
 \qbezier(5,0)(4.1,-1.2)(3.2,-2.4)
 }
 \put(80,00){
 {\ffv}
 \linethickness{.75pt}
 \qbezier[22](0,0)(-8,7)(-16,14)
 \qbezier[33](0,40)(-8,27)(-16,14)
 \qbezier[22](0,40)(8,33)(16,26)
 \qbezier[33](0,0)(9,13)(16,26)
 \linethickness{.5pt}\setlength{\unitlength}{12pt}
 \qbezier(0,0)(-1.2,0.9)(-2.4,1.8)
 \qbezier(0,5)(-1.2,3.4)(-2.4,1.8)
 \qbezier(0,0)(1.2,1.6)(2.4,3.2)
 \qbezier(0,5)(1.2,4.1)(2.4,3.2)
 }
 \put(80,-40){
 {\ffv}
 \linethickness{.75pt}
 \qbezier[22](0,0)(-8,7)(-16,14)
 \qbezier[33](0,40)(-8,27)(-16,14)
 \qbezier[22](0,40)(8,33)(16,26)
 \qbezier[33](0,0)(9,13)(16,26)
 \linethickness{.5pt}\setlength{\unitlength}{12pt}
 \qbezier(0,0)(-1.2,0.9)(-2.4,1.8)
 \qbezier(0,5)(-1.2,3.4)(-2.4,1.8)
 \qbezier(0,0)(1.2,1.6)(2.4,3.2)
 \qbezier(0,5)(1.2,4.1)(2.4,3.2)
 }
 \put(120,-40){
 {\ffv}
 \linethickness{.75pt}
 \qbezier[22](0,0)(-8,7)(-16,14)
 \qbezier[33](0,40)(-8,27)(-16,14)
 \qbezier[22](0,40)(8,33)(16,26)
 \qbezier[33](0,0)(9,13)(16,26)
 \linethickness{.5pt}\setlength{\unitlength}{12pt}
 \qbezier(0,0)(-1.2,0.9)(-2.4,1.8)
 \qbezier(0,5)(-1.2,3.4)(-2.4,1.8)
 \qbezier(0,0)(1.2,1.6)(2.4,3.2)
 \qbezier(0,5)(1.2,4.1)(2.4,3.2)
 }
 \put(80,-40){
 {\ffh}
 \linethickness{.75pt}
 \qbezier[22](0,0)(7,8)(14,16)
 \qbezier[33](40,0)(27,8)(14,16)
 \qbezier[22](40,0)(33,-8)(26,-16)
 \qbezier[33](0,0)(13,-8)(26,-16)
 \linethickness{.5pt}\setlength{\unitlength}{12pt}
 \qbezier(0,0)(0.9,1.2)(1.8,2.4)
 \qbezier(5,0)(3.4,1.2)(1.8,2.4)
 \qbezier(0,0)(1.6,-1.2)(3.2,-2.4)
 \qbezier(5,0)(4.1,-1.2)(3.2,-2.4)
 }
 \put(120,00){
 {\ffh}
 \linethickness{.75pt}
 \qbezier[22](0,0)(7,8)(14,16)
 \qbezier[33](40,0)(27,8)(14,16)
 \qbezier[22](40,0)(33,-8)(26,-16)
 \qbezier[33](0,0)(13,-8)(26,-16)
 \linethickness{.5pt}\setlength{\unitlength}{12pt}
 \qbezier(0,0)(0.9,1.2)(1.8,2.4)
 \qbezier(5,0)(3.4,1.2)(1.8,2.4)
 \qbezier(0,0)(1.6,-1.2)(3.2,-2.4)
 \qbezier(5,0)(4.1,-1.2)(3.2,-2.4)
 }
 \setlength{\unitlength}{6pt}
 \qbezier[22](0,0)(7,8)(14,16)
 \qbezier[33](40,0)(27,8)(14,16)
 \qbezier[22](40,0)(33,-8)(26,-16)
 \qbezier[33](0,0)(13,-8)(26,-16)
 \linethickness{.5pt}\setlength{\unitlength}{48pt}
 \qbezier(0,0)(0.9,1.2)(1.8,2.4)
 \qbezier(5,0)(3.4,1.2)(1.8,2.4)
 \qbezier(0,0)(1.6,-1.2)(3.2,-2.4)
 \qbezier(5,0)(4.1,-1.2)(3.2,-2.4)
 \end{picture}
 \caption{An illustration for the Open Set Condition in the case of the square Koch curve, also called the Minkowski fractal. The thick dotted line outlines the set $\mathcal{O}$, which is called the 0-cell. The thin dotted lines outlines the 
 open sets in $\Phi_1(\mathcal{O})$, which are called 1-cells. 
 The bottom picture illustrates the stronger form of the Open Set Condition used in Conjecture~\ref{con-sOSC}: the thin solid lines outline the 
 open sets $ \mathcal{O}'$ and $\Phi_1(\mathcal{O}')$. 
 }\label{fig-OSC}
 \end{figure}
 
 We now discuss the more general set-up in Subsection~\ref{subsecOSC}. 
 The standard Open Set Condition \cite[Section 9.2]{Falconer} is satisfied for an iterated function system if there is a non-empty bounded open set $O$ such that $\Phi_m(O)\subset O$ with the union in the left-hand side disjoint. Note that the open set $O$ may not be unique. 
 Conjecture~\ref{con-sOSC} assumes The Fractal Self-Similar Face Condition (Assumption \ref{a-face}) and a strong version of the Open Set Condition, see Figure~\ref{fig-OSC}, that we introduce as follows.

 \begin{assumption}[A Strong Open Set Condition]\label{a-OSC}
 We assume the Open Set Condition for the sequence $\Phi_m$ 
 is satisfied with 
 two different convex open polygons $\mathcal{O}\subsetneqq\mathcal{O}'$,
 not depending on $m$, 
 such that 
 $$\partial {\mathcal{O}}\cap K_0=\partial {\mathcal{O}}'\cap K_0=\partial {\mathcal{O}}\cap 
 \partial {\mathcal{O}}'=\partial_{(n-2)}K_0.$$
 \end{assumption} 
 
 \begin{conjecture}\label{con-sOSC}:
 \emph{If Assumptions~\ref{a-face} and \ref{a-OSC} are satisfied, 
 then $\Omega_m$ and $\Omega$ are 
 uniformly 
 exterior and interior 
 $(\epsilon,\infty)$-domains, that is, $\epsilon$ does not depend on $m$.}
 \end{conjecture}
 
 One possible approach to this conjecture, following Definition~\ref{DefEDD} of an $(\eps,\delta)$-domain from \cite{JONES-1981} and \cite[Remark 1]{HINZ-2021-1}, condition (ii) is equivalent to saying that the \emph{$\frac{1}{\varepsilon}$-cigar} 
 \begin{equation}\label{E:cigar}
 C(\gamma,\varepsilon):=\bigcup_{z\in\gamma} B(z,\varepsilon\lambda(z)),\quad \text{where}\quad \lambda(z)=|x-z|\frac{|y-z|}{|x-y|},\quad z\in\gamma,
 \end{equation}
 is contained in $\Omega$. 
 Another possible approach to this conjecture is by the recent result~\cite[Theorem 2.15]{AHMNT} (see also~\cite[Appendix A]{AHMNT}), 
 it is enough to prove the interior and exterior NTA conditions with uniform constants (we do not provide here the definition of an NTA domain, see \cite{JERISON-1982} or  \cite{AHMNT}). 
 By~\cite[Definition 2.12]{AHMNT}, we need to verify the Corkscrew condition~\cite[Definition 2.10]{AHMNT}
 and 
 the Harnack chain condition~\cite[Definition 2.12]{AHMNT}
 with constants not depending on $m$.
 The Corkscrew condition, both exterior and interior, is immediately implied by the self-similarity and the Strong Open Set Condition. The essential arguments in the proof of 
 the Harnack chain condition
 are similar to those in Ref.~\cite{ALM2003}, where the reader can find background and detailed explanations of the techniques. 
 
 In the two dimensional case there are 
 more straightforward arguments 
 to show that polygonal approximations to a self-similar curve bound uniformly $(\eps,\infty)$-domains. 
 Such arguments can be based on the Ahlfors three point condition, see~\cite[page 73]{JONES-1981}. 

\def\refname{References}
\bibliographystyle{siam}
\bibliography{ref}
\end{document}